\documentclass[letterpaper,11pt]{article}

\usepackage{color}
\usepackage{grffile}
\usepackage{ctable}
\usepackage{amssymb}
\usepackage{amsmath,color}
\usepackage[dvips]{epsfig}
\usepackage[small]{caption}
\usepackage{graphicx}
\usepackage[letterpaper,text={16cm,23cm},centering]{geometry}

\usepackage{tikz}
\usepackage{txfonts}
\usetikzlibrary{matrix,arrows,decorations.pathmorphing}

\newtheorem{lem}{Lemma}[section]
\newtheorem{thm}{Theorem}
\newtheorem{prop}[lem]{Proposition}
\newtheorem{cor}[lem]{Corollary}
\newtheorem{rem}[lem]{Remark}

\newenvironment{proof}[1][\unskip]%
 {\begin{trivlist} \item[]{\bf Proof #1. }}%
 {\hspace*{\fill}$\rule{.4\baselineskip}{.4\baselineskip}$\end{trivlist}}

 {\begin{trivlist}\item[]\textbf{Acknowledgments.}}{\end{trivlist}}

 {\begin{trivlist} \item[]{\bf\color{red} Comment. }}%
 {\color{red}\hspace*{\fill}$\rule{.4\baselineskip}{.4\baselineskip}$\end{trivlist}}

\makeatletter\@addtoreset{figure}{section}\makeatother

\makeatletter \@addtoreset{equation}{section} \makeatother

\DeclareMathOperator\arctanh{arctanh}

\newcommand{\rmO}{\mathrm{O}}

\newcommand{\rmd}{\mathrm{d}}
\newcommand{\rme}{\mathrm{e}}
\newcommand{\rmi}{\mathrm{i}}

\newcommand{\Rg}{\mathrm{Rg}\,}
\renewcommand{\Im}{\mathrm{Im}\,}

\renewcommand{\leq}{\leqslant}
\renewcommand{\geq}{\geqslant}

\def\XXint#1#2#3{{\setbox0=\hbox{$#1{#2#3}{\int}$}
     \vcenter{\hbox{$#2#3$}}\kern-.5\wd0}}

\newcommand{\eps}{\varepsilon}
\newcommand{\R}{\mathbb{R}}

\newcommand{\bbP}{\mathbb{P}}
\newcommand{\bbR}{\mathbb{R}}

\newcommand{\bbN}{\mathbb{N}}
\newcommand{\bbZ}{\mathbb{Z}}

\newcommand{\mmT}{\mathcal{T}}

\newcommand{\mmL}{\mathcal{L}}

\newcommand{\mmX}{\mathcal{X}}

\title{Pacemakers in large arrays of oscillators with nonlocal coupling}

\author{Gabriela Jaramillo and Arnd Scheel
\footnote{GJ and AS acknowledge support by the National Science Foundation through grants DMS-0806614 and DMS-1311740.}
 \\
University of Minnesota \\
School of Mathematics \\
127 Vincent Hall, 206 Church St SE \\
Minneapolis, MN 55455, USA}

\begin{document}
\maketitle

\begin{abstract}
We model pacemaker effects of an algebraically localized heterogeneity in a 1 dimensional array of oscillators with nonlocal coupling. We assume the oscillators obey simple phase dynamics and that the array is large enough so that it can be approximated by a continuous nonlocal evolution equation. We concentrate on the case of heterogeneities with positive average and show that steady solutions to the nonlocal problem exist. In particular, we show that these heterogeneities act as a wave source, sending out waves in the far field. This effect is not possible in 3 dimensional systems, such as the complex Ginzburg-Landau equation, where the wavenumber of weak sources decays at infinity. To obtain our results we use a series of isomorphisms to relate the nonlocal problem to the viscous eikonal equation. We then use Fredholm properties of the Laplace operator in Kondratiev spaces to obtain solutions to the eikonal equation, and by extension to the nonlocal problem.
\end{abstract}

\section{Introduction}
 
  The collective behavior in systems of coupled oscillators has attracted a tremendous amount of interest. Self-organized synchronization in large systems appears particularly dramatic when coupling effects are seemingly weak \cite{strogatzbook}. A substantial part of the work has been  devoted to the study of such collections of oscillators in the strikingly simple and explicit Kuramoto model \cite{kuramoto}. Synchronization and desynchronization as well as a plethora of more complex states have been found, and boundaries (or phase transitions) have been characterized \cite{strogatz1, chiba}. On the other hand, it is well known that the collective behavior may well depend on the type of coupling, as well as the type of internal dynamics at nodes. Of particular interest have been spiky oscillators in neuroscience with their quite different phase response curves, or phase-amplitude descriptions near Hopf bifurcations. 

Our interest here focuses modestly on a rigorous description of pacemakers. Roughly speaking, we ask if and how a small collection of oscillators can influence the collective behavior of a large ensemble. This question has been addressed in numerous contexts. One observed dramatic influence manifests itself through the occurrence of target patterns. Phenomenologically, a faster (or slower) patch of oscillators entrains neighbors and a phase-lag gradient propagates through the medium according to an eikonal equation. 

The anlaysis of such phenomena is notoriously complicated by the absence of spectral gaps in the linearization at the synchronized state, inherently related to the presence of a neutral phase in the medium. Standard perturbation analysis in a large collection of oscillators, based on an Implicit Function Theorem, is valid only for extremely small sizes of perturbations and fails to capture key phenomena. In an infinite medium, the range of the linearization is not closed, so that simple matched asymptotics cannot be justified. In fact, in an infinite medium (and also, with some corrections, in large media), one observes that the system relaxes to a frequency-synchronized state, but the collective frequency depends in unusual ways on the perturbation parameter. Characterizing, for instance, the deviation of the localized patch of oscillators from the ensemble background by $\varepsilon$, the collective frequency will change with $\omega\sim \varepsilon^2$ for $\varepsilon>0$ and remain constant for $\varepsilon<0$, in a one-dimensional medium. It is this general phenomenon that we are concerned with in this paper. 

One can ask questions of perturbative nature in many different circumstances. First, one can consider various types of oscillators, ranging from simple phase oscillators $\phi'=\omega$, over gauge-invariant phase-amplitude oscillators, $A'=(1+\rmi\omega)A-(1+\rmi\gamma)A|A|^2$, to general asymptotically stable periodic orbits $u_*(-\omega t)$ in an ODE $u'=f(u)$. On the other hand, one can look at simple scalar diffusive coupling, or, most generally, dynamics on networks. Our focus is on \emph{simple} internal phase dynamics, but non-local coupling along a line. Previous results have studied phase-dynamics, formally derived from the complex Ginzburg-Landau equation for amplitude-phase oscillators \cite{mikh}, and general stable periodic orbits with diffusive coupling, but in a one-dimensional context \cite{ssdefect} or with radial symmetry \cite{kollarscheel}. Radial symmetry was removed as an assumption in \cite{jaramillo2} in the complex Ginzburg-Landau equation in 3 dimensions. 

Phase dynamics can be derived and shown to approximate dynamics on long temporal and spatial scales \cite{dsss}. A general form of the dynamics is 
\begin{equation}\label{e:eik}
\phi_t=d\Delta \phi -\kappa|\nabla \phi|^2 + \omega_*, \qquad \phi\in \bbR/(2\pi\bbZ).
\end{equation}
Substituting $\phi\to \phi+\omega_* t$, thus exploiting the phase invariance, we can assume that $\omega_*=0$. The solutions $\phi\equiv \bar{\phi}$ correspond to the spatially synchronized state. One can show that this synchronized state is asymptotically stable under localized ($L^1$) perturbations, with decay rate given by the ``effective viscosity''  $\sim  dt^{-n/2}$ \cite{GS}.

Adding a localized inhomogeneity,
\begin{equation}\label{e:eikinh}
\phi_t=d\Delta \phi -\kappa|\nabla \phi|^2 + \omega_*+  \varepsilon g(x), \qquad g(x)\to 0 \mbox{ for } |x|\to\infty,
\end{equation}
will destroy the synchronized state. In particular, $g>0$, compactly supported, encodes a localized patch of oscillators oscillating at a higher frequency $\omega_*+\eps g(x)$. One can then ask if the system (\ref{e:eikinh}) possesses a periodic solution for $\varepsilon\neq 0$, and what the frequency of this solution would be. Therefore, note first that (\ref{e:eikinh}) possesses nontrivial solutions at $\varepsilon=0$, 
\[
\phi(t,x)=k\cdot x-\omega t, \qquad \omega= \kappa |k|^2.
\]
This family of solutions mimics periodic wave-trains $u_*(-\omega t+kx)$ in a reaction-diffusion context \cite{ssdefect,kollarscheel} or plane-wave solutions $\sqrt{1-k^2}\rme^{\rmi(-\omega t + k\cdot x)}$ in the complex Ginzburg-Landau equation. The coefficient $\kappa$ (which could be scaled to $\kappa=1$) therefore encodes nonlinear dispersion, that is, the dependence of nonlinear (here, affine) wave frequency on the spatial wavenumber. These waves travel in the direction of $k=\nabla\phi$, with group velocity $c_\mathrm{g}=2\kappa k$. We are therefore interested in solutions $\phi(x-\omega t)$ for which $c_\mathrm{g}=2\kappa\nabla\phi(x)\cdot x\geq 0$ for $|x|\to \infty$. In other words, we focus on waves ``generated'' by the inhomogeneity, rather than the effect of the inhomogeneity as a scattering object on waves sent in from infinity. 

Equation (\ref{e:eikinh}) can be analyzed using a variety of methods. Cole-Hopf will conjugate the equation to a Schr\"odinger eigenvalue problem, where small eigenvalues can pop out of the edge of the essential spectrum depending on the sign of $\varepsilon \int g$ \cite{kollarscheel}. Studying eigenvalues of Schr\"odinger operators opens up an entirely different set of tools; see \cite{simon}. An approach that carries over to more general one-dimensional systems relies on rewriting (\ref{e:eikinh}) as an ODE,
\begin{align}
\phi_x&=u\\
u_x&=\frac{\kappa}{d}u^2-\omega-\varepsilon g(x).
\end{align}
When $g$ is exponentially localized, one can compactify space $x=\arctanh \tau$, and study heteroclinic orbits in 
\begin{align}
u_x&=\frac{\kappa}{d}u^2-\omega-\varepsilon g(x(\tau))\\
\tau_x&=1-\tau^2.
\end{align}
We refer to \cite{ssdefect,kollarscheel} for discussions of the corresponding heteroclinic bifurcation. Radial symmetry allows for a similar approach based on dynamical systems methods. 

Our main result is concerned with the nonlocal equivalent of (\ref{e:eikinh}), 
\begin{equation}\label{e:eiknl}
\phi_t=-\phi+G\ast\phi - (J'\ast\phi)^2+\varepsilon g(x).
\end{equation}
Here, $G$ and $J$ are symmetric convolution kernels, $x\in \bbR$, and $\int G=1$. We are motivated by two aspects. First, nonlocal coupling is more realistic in most examples of coupled oscillator problems; nonlocal kernels arise naturally in the limit of large networks \cite{medvedev}. Second, nonlocal problems are not immediately amenable to the type of dynamical systems approach described above and therefore pose interesting technical challenges. Indeed, the linearization at $\phi\equiv \bar{\phi}$ is not Fredholm as a closed operator on typical spaces and a more subtle analysis is necessary. 

Linear problems similar to (\ref{e:eiknl}) have been studied in \cite{fayescheelindiana} using spaces of exponentially localized functions. As demonstrated in \cite{jaramillo2,jaramillo1}, this approach is not easily viable in higher space dimensions. A more robust approach relies on algebraic weights and will be pursued here. As a side benefit, we are also able to allow algebraically localized inhomogeneities $g$. Such weak algebraic localization would cause problems even in the local version, since time compactifications would leave equilibria at infinity highly degenerate, necessitating for instance refined geometric blow-up methods. 

We start in Section \ref{section:WeightedSpaces} with an overview of algebraically weighted spaces and Fredholm properties of the Laplacian and related operators. The goal is then to use this information in Section \ref{section:Eikonal} to analyze equation \eqref{e:eikinh} when the inhomogeneity, $g$, is assumed to be algebraically localized. This represents a simpler version of the nonlocal case we want to understand, since the linearizations of equations \eqref{e:eikinh} and  \eqref{e:eiknl} about the constant solution share the same Fredholm properties. Finally, in Section \ref{section:Nonlocal} the procedures used to find solutions for the local case are extended  to solve the nonlocal problem with the following assumptions:

\begin{itemize}
\item[{\bf H1}] \textbf{Diffusive Kernel:} The kernel $G$ is continuous, even, exponentially localized with 
\[
\int G(x)\rmd x=1,\quad G_2:=\int x^2G(x)\rmd x>0.
\]
Moreover, we require that the Fourier transform satisfies $\hat{G} -1 \leq 0$, which encodes linear stability of the synchronous state at $\eps=0$. 
\item[{\bf H2}] \textbf{Non-local Transport:} The kernel $J$ is exponentially localized, even, twice continuously differentiable, and 
\[ 
J_0:=\int J(x)\rmd x\neq 0.
\]
\item[{\bf H3}] \textbf{Inhomogeneity:} The function $g$ is algebraically localized, that is, for some $\sigma>2$, we have
\[
\int(\partial^jg(x))^2(1+x^2)^{\sigma+4}\rmd x <\infty,\quad \mbox{ for } j=0,1,2.
\]
Moreover, we assume that
$
g_0:=\int g(x)\rmd x\neq 0
,$ and define $g_1:=\int xg(x)\rmd x.$
\end{itemize}
Note that \eqref{e:eiknl} possesses wave train solutions of the form $\phi(t,x)=kx-\omega t$, when the nonlinear dispersion relation
\[
\omega=\omega_\mathrm{nl}(k):=J_0^2 k^2
\]
is satisfied. We are interested in pacemakers (or sources), which, according to the above discussion and \cite{dsss}, resemble wave trains at $\pm\infty$ with outward pointing group velocity \cite{ssdefect,kollarscheel}
\[
\pm c_\mathrm{g}^\pm>0,\qquad \mbox{where }c_g^\pm=2 J_0^2 k_\pm.
\]
We are now ready to state the main result. 

\begin{thm}\label{thm:Nonlocal}
Consider the nonlocal eikonal equation \eqref{e:eiknl}, under Hypotheses (H1)--(H3). Then, there exists $\eps_0>0$ such that for all $0<|\eps|<\eps_0$ and $\mathrm{sign}(\eps)=-\mathrm{sign}(g_0)$, there exists a  solution of the form 
\[
\Phi(x,t; \eps) = \phi(x,\eps) + \left(\phi_0(\eps)+ k(\eps)x\right)\tanh(x)-\omega_\mathrm{nl}(k(\eps)) t,
\]
where $\phi_0,k$ are $C^1$ with 
\[ \phi_0'(0) = \frac{g_1}{G_2},\qquad k'(0)= -\frac{g_0}{G_2}, \]
and
\[
|\phi(x;\eps)|\to 0, \mbox{ for } |x|\to\infty,\mbox{ uniformly in } \eps.
\]
\end{thm}
Note that the sign of $\partial_x\Phi$ is such that group velocities point outwards in the far field, since $\partial_x\Phi\sim \pm\eps k'(0)=\pm\eps g_0/G_2>0$ for $x\to\pm\infty$.

More precise statements on the dependence of $\phi$ on $\eps$ and $x$ can be found in the proof. For instance, $\phi$ is $C^1$ in $\eps$ with values in $M^{2,2}_{\sigma}$, a Kondratiev space that we shall define below. One could also obtain higher-order expansions in $\eps$ by assuming more localization on $g$ as we shall see from the proof.  Additionally, the result readily generalizes to more general nonlinearities $f(J'\ast\phi)$, $f(u)=\rmO(u^2)$. 

On the other hand, we do not believe that our assumptions on localization are sharp. But then again, critical localization is not fully understood, even in the simple conjugate problem of Schr\"odinger eigenvalue bifurcations from the essential spectrum; see for instance \cite{sscrit} and references therein.

\paragraph{Outline.} The remainder of this paper is organized as follows. We first introduce function spaces that will be used throughout the remainder of this paper and state some basic results on Fredholm properties of differential operators in Section \ref{section:WeightedSpaces}. Section \ref{section:Eikonal} is concerned with a warm-up problem: we prove Theorem \ref{thm:Nonlocal} in the (local) case of the eikonal equation, replacing nonlocal oprators by differential operators as in \eqref{e:eik}. We then move to the proof of Theorem \ref{thm:Nonlocal} in Section \ref{section:Nonlocal}. The Appendix contains some rather elementary results on Fredholm properties of one-dimensional differential operators in Kondratiev spaces that we were unable to locate in the literature. 

 \section{Weighted and Kondratiev spaces} \label{section:WeightedSpaces}
We give a short summary of the weighted spaces that we will be using throughout. We also collect Fredholm properties of various differential operators. 
\subsection{Algebraic weights}
We use weighted Lebesgue $L^p$- and Sobolev $W^{k,p}$-spaces on the real line, defined as completions of $C^\infty_0$ under the norms
\begin{equation*}
\|u\|_{L^p_\gamma}=\left(\int |u(x)\langle x \rangle^{\gamma}|^p     \right)^{1/p},\qquad
\|u\|_{W^{k,p}_\gamma}=\left(\sum_{0\leq \alpha\leq k}\|\partial_x^\alpha u\|_{L^p_\gamma}\right)^{1/p},
\end{equation*}
where $\langle x \rangle = ( 1+|x|^2)^{1/2}$, $1\leq p<\infty$, and $k\in\bbN$. Note that  $W_{\beta}^{k.p}\subset W^{k,p}_{\alpha}$ whenever  $ \alpha < \beta$.  Also note that for $p=2$, the Fourier transform maps $\mathcal{F}:W^{s,2}_{\gamma}\to W^{\gamma,2}_{s}$, where one can consider this mapping as a definition of the fractional Sobolev spaces $W^{\gamma,2}$. We also write $W^{-k,q}_{-\gamma}$, $1/p+1/q=1$, for the dual of $W^{k,p}_\gamma$, so that the statement on Fourier images can be extended to negative values of $\gamma$. The following result is routine. 
\begin{prop}\label{prop:(1-dxx)}
The operators $1-\partial_{xx} : W_{\gamma}^{k+2,p} \rightarrow W^{k,p}_{\gamma}$ and $1-\partial_{x} : W_{\gamma}^{k+1,p} \rightarrow W^{k,p}_{\gamma}$ are isomorphism for all $p\in(1,\infty)$ and all $\gamma\in\R$.
\end{prop}
\begin{proof}
Multiplication by $\langle {\bf x} \rangle^{\gamma}$ provides an isomorphism between $W^{k,p}_\gamma$ and $W^{k,p}$. Since commutators of this multiplication operator and derivatives are relatively compact operators, and since the kernel is always trivial,  one readily concludes that it is sufficient to prove the result for $\gamma=0$, where in turn the inverse is explicitly known. 
\end{proof}

%
%
We remark that $\partial_{x}^k:W^{k,p} \rightarrow L^p$ does not have closed range, as an explicit Weyl sequence construction shows. 
An argument as in the preceding proof implies the same statement for $\partial_{x}^k:W^{k,p}_{\gamma} \rightarrow L^p_{\gamma}$.

\subsection{Kondratiev spaces}
Defining a norm for which localization increases with each derivative can help recover Fredholm properties. We define the Kondratiev spaces as completions in $C^\infty_0$ under the norm
\[ 
\|u \|_{M^{k,p}_{\gamma}} = \left( \sum_{|\alpha|\leq k} \| \partial_x^{\alpha} u \cdot \langle x \rangle^{\gamma+|\alpha|} \|^p_{L^p} \right)^{1/p}.
\]
We will denote the dual space to $M^{k,p}_{\gamma}$ by $M^{-k,q}_{-\gamma}$, where $1/p+1/q=1$.
Such spaces have been used extensively in regularity theory \cite{kondrat1967boundary}, fluids \cite{melcher2008hopf,specovius1986exterior}, and in our prior work on inhomogeneities \cite{jaramillo1,jaramillo2}. Fredholm properties for the Laplacian and various generalizations have been established in \cite{nirenberg1973null,lockhart1981fredholm, lockhart1983elliptic, lockhart1985elliptic}. We will rely on the following, much more elementary result. 
%
%
 \begin{prop}\label{prop:FredholmCompDerivatives}
 Let $m$ and $l $ be non negative integers, and $p \in (1, \infty)$. Then, the operator $$ (1-\partial_x)^{-\ell}\partial_x^m :M^{m,p}_{\gamma-m} \rightarrow W^{\ell,p}_{\gamma},$$ is Fredholm for $\gamma+1/p\not\in \{1,\ldots,m\}$. In particular,
 \begin{itemize}
 \item for $\gamma < 1-1/p$ it is surjective with kernel spanned by $\bbP_m$;
 \item for $\gamma >m-1/p$ it is injective with cokernel spanned by $\bbP_m$;
 \item for $j-1/p< \gamma< j+1-1/p$, where $j \in \bbN$, $1 \leq j <m$, its kernel is spanned by $\bbP_{m-j}$ and its cokernel is spanned by $\bbP_j$.
 \end{itemize}
For  $\gamma+1/p\in \{1,\ldots,m\}$, the operator does not have close range.  Here, $\bbP_m$ is the $m$-dimensional space of all polynomials with degree less than $m$. 
\end{prop}
We give a proof of this proposition in the Appendix. 
%

\section{The Eikonal Equation --- a Local Warm-Up}\label{section:Eikonal}
We consider the local eikonal equation 
\begin{equation}\label{eq:peikonal}
\phi_t = \partial_{xx} \phi -( \partial_x \phi)^2+ \eps g(x), \quad x \in \bbR,\quad \eps>0,
\end{equation} 
and look for wave sources in the spirit of Theorem \ref{thm:Nonlocal}. Of course, wave trains $\phi=kx-\omega t$ exist for $\eps=0$ and when $\omega=\omega_\mathrm{nl}(k)=k^2$. 
\begin{thm}\label{thm:Eikonal}
Consider  \eqref{eq:peikonal}, under Hypothesis (H3). Then, there exists $\eps_0>0$ such that for all $0<|\eps|<\eps_0$ and $\mathrm{sign}(\eps)=- \mathrm{sign}(g_0)$, there exists a  solution of the form 
\[
\Phi(x,t; \eps) = \phi(x,\eps) +\left(\phi_0(\eps)+ k(\eps)x\right)\tanh(x)-\omega_\mathrm{nl}(k(\eps)) t,
\]
where $\phi_0,k$ are $C^1$ with 
\[ \phi_0'(0) = \frac{1}{2}{g_1},\qquad k'(0)=-\frac{1}{2} {g_0}, \]
and
\[
|\phi(x;\eps)|\to 0, \mbox{ for } |x|\to\infty,\mbox{ uniformly in } \eps.
\]
\end{thm}
\begin{rem}
This result can be obtained directly using the Hopf-Cole linearizing transformation and results on eigenvalues of Schr\"odinger operators with small potentials; see \cite{simon,kollarscheel}. Our proof here is significantly more involved but lays the ground for the nonlocal result, Theorem \ref{thm:Nonlocal}.
\end{rem}

The proof of Theorem \ref{thm:Eikonal} is organized as follows. We find first-order approximations to solutions of \eqref{eq:peikonal} in Section \ref{subsec:O(1)}. In Section \ref{subsec:OperatorFb}, we formulate the problem of existence of $\Phi$ as finding the zeros of a nonlinear operator $F_{b}:M^{2,p}_{\sigma} \times \bbR^2 \rightarrow M^{2,p}_{\sigma} \times \bbR^2$, which we construct as the composition of a linear preconditioner $\mmT_{b}^{-1}$, and nonlinear operators $\tilde{N}_1$ and $\tilde{N}_2$. The additional variables $(a,b)\in\R^2$ stand for explicit far-field corrections, which will yield the terms including $\phi_0$ and $k$ in the expansion. We make $b$, which accounts for the most dramatic correction, explicit as an index. Our strategy is to show that $F_b$ satisfies the conditions of the Implicit Function Theorem.  We therefore show bounded invertibility of  $\mmT_b$ in Section \ref{subsec:prelim} for $b\geq 0$. In Sections \ref{subsec:ContinuousTb} and \ref{subsec:ContinuousbTb}, we 
prove that $\mmT_b^{-1}$ and $b\mmT_b^{-1}$ are smooth and continuously differentiable with respect to the parameter $b$, for $b\geq 0$. We finally combine these results in Section \ref{subsec:ProofEikonal} to prove Theorem \ref{thm:Eikonal}. In the following, we will always assume $g_0=\int g<0$ and choose $\eps>0$.
\subsection{First order approximations}\label{subsec:O(1)}
To start with, notice that the linearization of \eqref{eq:peikonal} about the constant solution results in the Laplace operator, which, according to  Proposition \ref{prop:FredholmCompDerivatives} is Fredholm for $\gamma > 2- 1/p$ on
\[ 
\partial_{xx} : M^{2,p}_{\gamma-2} \rightarrow L^p_{\gamma},
\]
with index $-2$ and cokernel spanned by $\{ 1, x\}$. We therefore insert the Ansatz,
\begin{equation*}
\phi(x,t) = \tilde{\phi}(x) + a S(x) + b x S(x) - b^2 t, \quad S(x) = \tanh(x), \ a, b \in \bbR, 
\end{equation*}
into  \eqref{eq:peikonal} and obtain, dropping tildes,
\begin{equation}\label{eq:eikonal2}
0= \partial_{xx}\phi -2bS \partial_x\phi+ a\partial_{xx}S + b \partial_{xx}(xS)+ \eps g- N(\phi, a, b),
\end{equation}
where we gathered all nonlinear terms in 
\begin{equation}\label{eq:nonlinearities}
 N(\phi,a,b) = (\partial_x \phi)^2 - b^2(1 -S^2) + 2( \partial_x \phi + bS)(a + bx) \partial_x S + ( a+bx)^2(\partial_x S)^2.
 \end{equation}
We would like to consider the right-hand side of  \eqref{eq:eikonal2} as a map from $M^{2,2}_{\gamma-2}\times \bbR^2$ into $L^2_{\gamma}$ for some $\gamma>2-1/p$, and then use the Fredholm properties of the Laplacian along with Lyapunov-Schmidt reduction to find solutions. The main difficulty is that for $\phi \in M^{2,2}_{\gamma-2}$, we do not obtain $S\partial_x \phi\in L^2_{\gamma}$, so that the right-hand side is not well defined.
Nonetheless, it is still possible to formally find a first order approximation for the solution, inserting an expansion of the form $\phi= \eps \phi_1, a=\eps a_1, b=\eps b_1$. At leading order, we find
\begin{equation}
 \partial_{xx} \phi_1 + a_1 \partial_{xx} S + b_1\partial_{xx}(xS) =-g.
\label{e:1st} \end{equation} 
Computing the scalar products
\begin{equation}\label{e:SxS}
\langle \partial_{xx}S, 1\rangle = 0, \quad\langle \partial_{xx}S,  x\rangle =-2, \quad \langle \partial_{xx}(xS),1 \rangle = 2,\quad \langle \partial_{xx}(xS), x \rangle =0,
\end{equation}
shows that  $\partial_{xx}S$ and $\partial_{xx}(xS)$, span the cokernel of $\partial_{xx}: M^{2,p}_{\gamma-2} \rightarrow L^p_{\gamma}$, and we obtain the following result.
\begin{lem}\label{lem:InvertibleLaplace}
For any $\gamma>2-1/p$, the operator $A: M^{2,p}_{\gamma-2} \times \bbR^2 \rightarrow L^p_{\gamma}$ defined as 
\[
A(\phi,a,b) = \partial_{xx} \phi + a \partial_{xx}S + b\partial_{xx}(xS),
\]
is invertible.
\end{lem}
In particular, given $g \in L^p_{\gamma}$ we can solve \eqref{e:1st} and  find $(\phi_1, a_1,b_1)\in M^{2,p}_{\gamma-2} \times \bbR^2$. Taking scalar products of \eqref{e:1st} with $1,x$ and using \eqref{e:SxS}, we also obtain $a_1= \frac{1}{2}g_1$ and $b_1=- \frac{1}{2} g_0.$

\subsection{Construction of the nonlinear map $F_b$} \label{subsec:OperatorFb}

In order to construct the map $F_b:M^{2,p}_{\sigma} \times \bbR^2 \times[0,\infty) \rightarrow M^{2,p}_{\sigma} \times \bbR^2$, we first introduce the space 
\begin{equation}\label{e:D}
\mathcal{D}=\{ u \in M^{2,p}_{\sigma} : u_x \in L^p_{\sigma+2}\},
\end{equation}
and the linear operator $\mmT_b:\mathcal{D} \times \bbR^2 \rightarrow L^p_{\sigma+2}$, 
$$
\mmT_b ( \rho, \alpha, \beta)= \partial_{xx} \rho - 2b S \partial_x \rho + \alpha\partial_{xx}S + \beta \partial_{xx}(xS)
.$$ 
A short calculation shows that inserting an Ansatz 
\[
\phi = \eps( \phi_1 + \rho), \quad a = \eps( a_1+ \alpha), \quad b= \eps(b_1+ \beta),
\]
into \eqref{eq:peikonal} gives the equation
\begin{equation*} 
 \mmT_{\eps(b_1 + \beta)}( \rho, \alpha, \beta) - \eps \tilde{N}_1( \rho,  \alpha, \beta) - \eps \bar{b} \tilde{N}_2=0,
\end{equation*}
where, $\tilde{N}_2 = 2S \partial_x \phi_1$,  and the operator $\tilde{N}_1(\rho,\alpha,\beta)$ is defined in terms of  $N$ from \eqref{eq:nonlinearities},
\begin{equation*}
 \tilde{N}_1(\rho, \alpha, \beta) = N(\phi_1+\rho, a_1+ \alpha, b_1+\beta).
\end{equation*}
Now, suppose for the moment that, for $b$  fixed, $\mmT_{ b}: \mathcal{D} \rightarrow L^2_{\sigma+2}$ is bounded invertible. We may then precondition the equation with $\mmT_b$ and solve the equivalent system
\begin{equation}\label{eq:operatorFb}
F_{b} (\rho, \alpha, \beta;\eps)= [I  -\eps \mmT_{ b}^{-1}( \tilde{N}_1+b \tilde{N}_2)](\rho,\alpha,\beta)=0.
\end{equation}
In particular, if $F_{\eps(b_1+\beta)}:M^{2,2}_{\sigma} \times \bbR^2\times[0,\infty) \rightarrow M^{2,2}_{\sigma}\times \bbR^2$ meets the conditions of the Implicit Function Theorem, we can  conclude the existence of solutions to \eqref{eq:peikonal}. We therefore will show that the operators
\begin{enumerate}
\item  $\mmT_b^{-1}\tilde{N}_1: M^{2,p}_{\sigma} \times \bbR^2 \rightarrow M^{2,p}_{\sigma} \times \bbR^2$, and
\item  $b\mmT_b^{-1}\tilde{N}_2:M^{2,p}_{\sigma} \times \bbR^2 \rightarrow M^{2,p}_{\sigma} \times \bbR^2$ 
\end{enumerate}
are  continuously differentiable with respect to  $b$.  We start proving that  $\mmT_b:\mathcal{D} \rightarrow L^p_{\sigma+2}$ is invertible, next.

\subsection{Invertibility of $\mmT^{-1}_b:L^p_{\gamma} \rightarrow M^{2,p}_{\gamma-2} \times \bbR^2$ for $\gamma>2-1/p$}\label{subsec:prelim}
Consider $\mmL_b : \mathcal{D}\rightarrow L^p_{\gamma},$ defined through 
\begin{equation}\label{e:lb}
\mmL_b \rho = \partial_{xx} \rho - 2b S \partial_x \rho.
\end{equation}
We will see that $\mmL_b$ is Fredholm of index $-2$, and $\partial_{xx} S,\partial_{xx}(xS)$ span its cokernel.  Throughout, we let $\Rg(\mmL_b)$ and $\Im^\perp(\mmL_b)$ denote range and cokernel of $\mmL_b$, and let $P: L^p_{\gamma} \rightarrow \Rg(\mmL_b)$ be a  projection onto its range. Since we will be using linear Lyapunov-Schmidt reduction, it is useful to have an explicit definition of  $P$. Notice that $\Im^\perp(\mmL_b) \subset L^q_{-\gamma}$ is spanned by
\begin{equation}\label{e:cok}
\psi^*_1(x) = \cosh(x)^{-2b},\qquad   \psi_2^*(x)= \int_0^x \left( \dfrac{\cosh(x)}{\cosh(y)}\right)^{-2b} \rmd y,
\end{equation}
where $(\partial_x +2b S(x)) \psi_1^* (x)= 0, \quad ( \partial_x +2b S(x) ) \psi_2^*(x) = 1.$ Using brackets $\langle \cdot, \cdot \rangle$ to express the relation between $L^p_{\gamma}$ and its dual $L^q_{-\gamma}$,  we pick $\psi_1, \psi_2$ such that $ \langle \psi_i, \psi^*_j\rangle =\delta_{ij}$, $ i,j=1,2$ and find
\[ P u = u-  \langle u , \psi_1^* \rangle \psi_1 - \langle u, \psi_2^* \rangle \psi_2. \]

Notice as well that the functions $\psi_1^*(x)$ and $\psi_2^*(x)$ are in $L^q_{-\gamma}$ only for $b\geq 0$ --- we will not be able to extend this analysis to $b<0$. We are now ready to establish the Fredholm properties and bounds on inverses. 

\begin{lem}\label{lem:FredholmLb}
Let $p\in (1,\infty)$ and $\gamma > 1-1/p$. Then, the operator $\mmL_b: \mathcal{D} \rightarrow L^p_{\gamma}$,  defined in (\ref{e:lb}) is Fredholm index $-2$, its co-kernel is given in (\ref{e:cok}), and the solution to $\mmL_b u = f$ satisfies bounds  $ \|u \|_{\mathcal{D}} \leq \dfrac{C}{b} \|f\|_{L^p_{\gamma}}$, with $C$ independent of $b$ and $f\in\Rg(\mmL_b)$.
\end{lem}

\begin{proof} 
Since solutions to the ODE $\mmL_b \phi = 0$ are either constant or exponentially growing at infinity, $\mmL_b: \mathcal{D}  \rightarrow L^p_{\gamma}$ has trivial kernel for $\gamma>0$. We therefore only need to show that the range of $\mmL_b$ is  closed to conclude that it is a Fredholm operator. We therefore examine the explicit solution formula
\[ u(x) = 
\left \{ \begin{array}{c c c}
\int_{-\infty}^x \int_{-\infty}^y f(z)\left( \frac{\cosh(z)}{\cos(y)}\right)^{-2b} \rmd z \rmd y & \text{if}& x<0 \\
\int_{\infty}^x \int_{\infty}^y f(z)\left( \frac{\cosh(z)}{\cos(y)}\right)^{-2b} \rmd z \rmd y & \text{if}& x\geq0.
\end{array}\right.
\]
A direct calculation shows that the conditions $\langle f , \psi^*_i\rangle =0$, $i=1,2$, guarantee continuity of $u$ and $u_x$ at $x=0$. It remains to show that $u \in \mathcal{D}$. 
Therefore, we factor 
\[
\mmL_b u = (\partial_x - 2bS) \partial_x u = f
,\]
and show that  $(\partial_x - 2bS)^{-1}: \Rg(\mmL_b)\subset L^p_{\gamma} \rightarrow L^p_{\gamma}$ is bounded. Subsequently solving the Fredholm equation (Proposition \ref{prop:FredholmCompDerivatives}) $\partial_x u = ( \partial_x -2bS)^{-1} f$ gives a solution with $u_x,u_{xx}\in L^p_{\gamma}$
and, since $u \in L^p_{\gamma-1}\subset L^p_{\gamma-2}$, $u\in\mathcal{D}$.

Next, we establish uniform bounds $\| u\|_{\mathcal{D}} \leq \dfrac{C}{b} \| f\|_{L^p_{\gamma}}$.
For $x \geq 0$,
\[
|u_x(x)| = \left|\int^{\infty}_x f(y) \left( \frac{\cosh(y)}{\cosh(x)} \right)^{-2b} \rmd y\right|\leq \frac{1}{2}\int^{\infty}_x |f(y)|\rme^{-2b(y-x)}\rmd y,\]
which gives, using $\langle x\rangle^{\gamma} \langle y \rangle^{-\gamma} \leq \langle x-y \rangle^{|\gamma|}$ and  Young's inequality,
\begin{equation*}
\|u_x\|_{L^p_{\gamma}[0,\infty)} \leq  \frac{C}{b} \| f\|_{L^p_{\gamma}}.
\end{equation*}
A similar analysis in the case of $x<0$ shows the bound $\|u_x\|_{L^p_{\gamma}(-\infty,0]} \leq \frac{C}{b} \|f\|_{L^p_{\gamma}}$, and the lemma follows from Proposition \ref{prop:FredholmCompDerivatives}.
\end{proof}

\begin{lem}\label{lem:boundedInverseLb}
The operator $\mmL_b^{-1} : \Rg(\mmL_b) \subset L^p_{\gamma} \rightarrow M^{2,p}_{\gamma-2}$, is uniformly bounded in $b\geq0$ provided $\gamma >2- 1/p$. Explicitly, we have 
\[
\|u\|_{M^{2,p}_{\gamma-2}} \leq C \| f\|_{L^p_{\gamma}},
\]
for all $f\in \Rg(\mmL_b)$ and all solutions $\mmL_b u =f$.
\end{lem} 
\begin{proof}
We use the fact that the operator  $\partial_x^{-1}: M^{1,p}_{\gamma-1} \rightarrow M^{2,p}_{\gamma-2} $ is bounded for $\gamma>2-1/p$, with bound depending only on the weight $\gamma$, and write $\mmL_b u=(\partial_x - 2bS) \partial_x u$. The result then follows, once we show that the operator $( \partial_x -2b S )^{-1}: L^p_{\gamma} \rightarrow M^{1,p}_{\gamma-1}$ is uniformly bounded in $b$. Explicitly, we need to show that solutions to  $ (\partial_x -2b S )v= f$ satisfy
\begin{equation}\label{bound1}
\| v\|_{L^p_{\gamma-1}} \leq C \| f\|_{L^p_{\gamma}}.
\end{equation}
We establish this estimate for $x>0$, the other case being analogous. Set 
\[ 
x = e^{\tau} \quad \tau \in\R,\quad w = e^{\bar{\gamma} \tau} v(e^{\tau}), \quad g = e^{( \bar{\gamma} +1)  \tau} f,
\]
which gives 
\[
w_{\tau} -M(\tau) w = g,\qquad M(\tau) = \bar{\gamma}+ 2b S e^{\tau}>\bar{\gamma}.
\]
We find $w$ as
\[ 
w(\tau) = -\int^{\infty}_{\tau} g(s) e^{- \int^s_{\tau}M(\sigma) \rmd \sigma } \rmd s.
\]
Since  $- \int^s_{\tau}M(\sigma) \rmd \sigma \leq - \bar{\gamma} (s - \tau)$, we obtain, using again Young's inequality,
\begin{equation*} \| w (\tau)\|_{L^p} \leq  \left( \int_\R \left[ \int^{\infty}_{\tau}  |g(s) | e^{-\bar{\gamma} (s-\tau) } \rmd s\right]^p \rm dx  \right)^{1/p} \leq  \bar{\gamma}^{-1}\| g(\tau)\|_{L^p}.
\end{equation*}
Setting $\gamma-1 = \bar{\gamma} - \frac{1}{p} $ we find in the original variables
\[
\| v \|_{L^p_{\gamma-1}[0,\infty)} \leq C \| f\|_{L^p_{\gamma}[0, \infty)},
\]
which proves \eqref{bound1}.

Finally, since we can write $v_x = f + 2bS v$, and since we have the bound $ \| v\|_{L^p_{\gamma}} \leq \frac{C}{b} \|f \|_{L^p_{\gamma}}$ from Lemma \ref{lem:FredholmLb}, we are able to conclude that
\[ \| v_x\|_{L^p_{\gamma}} \leq \|f\|_{L^p_{\gamma}} + 2 b \| v \|_{L^p_{\gamma}} \leq C \|f\|_{L^p_{\gamma}}.\]
Consequently,
\[ \|v\|_{M^{2,p}_{\gamma-1}} \leq C \|f\|_{L^p_{\gamma}}.\]
where $C$ is a generic constant, independent of $b$, that can change from line to line. 
\end{proof}

We are now ready to show the invertibility of $\mmT_b:\mathcal{D} \times \bbR^2 \rightarrow L^p_{\gamma}$ with uniform bounds.
\begin{lem}\label{lem:InvertibleTb}
For $p \in (1, \infty)$ and $b\geq 0$, small,   $\mmT_b: \mathcal{D} \times \bbR^2 \rightarrow L^p_{\gamma}$, defined through, 
\[
\mmT_b(\rho, \alpha, \beta)= \mmL_b \rho+ \alpha \partial_{xx}S + \beta \partial_{xx}(xS),
\]
is invertible. Furthermore, solutions $(\rho, \alpha, \beta)$ to $\mmT_b(\rho, \alpha, \beta) = f$ satisfy, 
\[
\| (\rho, \alpha, \beta) \|_{\mathcal{D} \times \bbR^2} \leq \frac{C}{b} \| f\|_{L^p_{\gamma}}, 
\quad  
\| (\rho, \alpha, \beta) \|_{M^{2,p}_{\gamma-2}\times \bbR^2} \leq C \| f\|_{L^p_{\gamma}},
\]
for $\gamma>1-1/p$ and $\gamma >2-1/p$, respectively, with constant $C$ independent of $b$.
\end{lem}

\begin{proof}
From Lemma \ref{lem:FredholmLb} we know that $\mmL_b: \mathcal{D} \rightarrow L^p_{\gamma}$ is a Fredholm operator with index $i=-2$ and cokernel spanned by $\psi^*_1(x) = \cosh(x)^{-2b}$ and $\psi^*_2(x)= \int_0^x \left( \dfrac{\cosh(x)}{\cosh(y)} \right)^{-2b} \rmd y $. We will use this information together with the Bordering Lemma for Fredholm operators to show that $\mmT_b$ is invertible.

Let $R: \bbR^2 \rightarrow (\Im(\mmL_b))^{\perp}$ be defined as $R(\alpha, \beta) =\alpha \partial_{xx}S + \beta \partial_{xx}(xS)$ and write $$\mmT_b(\rho, \alpha,\beta)= \mmL_b \rho + R(\alpha, \beta).$$
Since $\partial_{xx}S$ and $\partial_{xx}(xS)$ are exponentially localized functions, the operator $R$ is well defined and bounded. Since scalar products of $\partial_{xx}S$ and $\partial_{xx}(xS)$ with $\psi_1^*$ and $\psi_2^*$ form an invertible $2\times 2$ matrix, the range of $R$ forms a complement to $\Rg(\mmL_b)$, and $\mmT_b:\mathcal{D} \rightarrow L^p_{\gamma}$ is invertible. 

To obtain the desired bounds on $\rho$, decompose $\mmT_b(\rho, \alpha, \beta) =f$ into
\begin{align*} 
P [ \mmL_b \rho + \alpha \partial_{xx}S + \beta \partial_{xx}(xS)] &= P f,\\
(1-P)[ \alpha \partial_{xx}S + \beta \partial_{xx}(xS)] & = ( 1-P) f.
\end{align*}
From the second expression we obtain bounds of the form, $$ | \alpha | \leq C \| f\|_{L^p_{\gamma}}, \quad |\beta|\leq C \|f\|_{L^p_{\gamma}}.$$ Then, the first equation and the results from Lemmas \ref{lem:FredholmLb} and  \ref{lem:boundedInverseLb} show that
\[ \| \rho \|_D \leq \frac{C}{b} \|\tilde{f} \|_{L^p_{\gamma}} \quad \text{and}  \quad \|\rho \|_{M^{2,p}_{\gamma-2}} \leq C \| \tilde{f}\|_{L^p_{\gamma}},\]
where $\tilde{f} = f - \alpha \partial_{xx}S- \beta \partial_{xx}(xS)$. Finally, the desired bounds on the solution $(\rho, \alpha, \beta)$ follow from applying the triangle inequality to $\tilde{f}$, the bounds on $|\alpha|$ and $|\beta|$, and the fact that the functions $\partial_{xx}S, \partial_{xx}(xS)$ are exponentially localized. 
\end{proof}
The above lemma shows that the operator $\mmT_b^{-1} : L^p_{\gamma} \rightarrow M^{2,p}_{\gamma-2} \times \bbR^2$ is bounded  linear and we have the following corollary to Lemma \ref{lem:InvertibleTb}. 
\begin{cor}\label{cor:smoothTb}
Let $b \geq 0$, $\gamma>2-1/p$, and $p \in (1, \infty)$. Then $\mmT_b^{-1}:L^p_{\gamma} \rightarrow M^{2,p}_{\gamma-2} \times \bbR^2$ is linear, uniformly bounded in $b$. 
\end{cor}
Roughly speaking, we have shown that we can achieve $b$-uniform bounds by giving away two degrees of localization. In the following, we show that giving away one or two more degrees of localization, we may even obtain continuity and differentiability in $b$. Eventually, we will compensate for the loss of localization by exploiting the fact that the nonlinerity gains localization.

\subsection{Differentiability of $\mmT_b^{-1}: L^p_{\gamma} \rightarrow M^{2,p}_{\gamma-4}\times \bbR^2$ for $\gamma>4-1/p$}\label{subsec:ContinuousTb}
We start by establishing continuity with respect to $b$.
\begin{lem}\label{lem:contTb}
Let $\gamma>3-1/p$, $b\geq 0$, and $p \in (1, \infty)$ then, the operator $\mmT_b^{-1}:L^p_{\gamma} \rightarrow M^{2,p}_{\gamma-3} \times \bbR^2$ is Lipshitz in  $b$ in the operator norm topology.
\end{lem}
\begin{proof}
We show that 
\[
\left \|  (\mmT_{b+h}^{-1} - \mmT_b^{-1})f  \right\|_{M^{2,p}_{\gamma-3} \times \bbR^2} \leq C|h|\,\|f\|_{L^p_{\gamma}}
.\]
Writing $(\rho, \alpha, \beta)|_b=\mmT_b^{-1}f$, $(\rho, \alpha, \beta)|_{b+h}=\mmT_{b+h}^{-1}f$, and $(\Delta \rho, \Delta \alpha, \Delta \beta) = (\rho, \alpha, \beta)|_{b+h}-(\rho, \alpha, \beta)|_b$, we have to show that
\[
\left \|  (\Delta \rho, \Delta \alpha, \Delta \beta) \right\|_{M^{2,p}_{\gamma-3} \times \bbR^2} \leq C|h|\,\|f\|_{L^p_{\gamma}}.
\]
A short calculation shows that 
\[
\mmT_b(\Delta \rho, \Delta \alpha, \Delta \beta)= -2h S\partial_x \rho|_{b+h},
\]
so that, using Lemma \ref{lem:InvertibleTb} with $\gamma-1>2 -1/p$, we find that 
\begin{equation}\label{eq:boundsrho}
\| (\Delta \rho, \Delta \alpha, \Delta \beta) \|_{M^{2,p}_{\gamma-3}\times \bbR^2} \leq 2 |h| C \| \partial_x \rho|_{b+h} \|_{L^p_{\gamma-1}} \leq 2  C|h| \,\|f\|_{L^p_{\gamma}},
\end{equation}
where the last inequality follows again from  $\mmT_{b+h}(\rho, \alpha, \beta) = f$ and Lemma \ref{lem:InvertibleTb} with $\gamma>2-1/p$. This proves continuity. 
\end{proof}
We next use a weaker topology to conclude differentiability.

\begin{lem}\label{lem:diffTb}
Let $\gamma>4-1/p$, $b\geq 0$, and $p \in (1, \infty)$. Then $\mmT_b^{-1}:L^p_{\gamma} \rightarrow M^{2,p}_{\gamma-4} \times \bbR^2$ is differentiable in $b$ with Lipshitz continuous derivative, with values in the operator norm topology.
\end{lem}

\begin{proof}
We abbreviate $R= \mmT_b^{-1}$ and define the candidate for the derivative,
\[
\partial_b R|_b f= 2\mmT_b^{-1}S \partial_x( \mmT_b^{-1})^1 f,
\]
where $(\mmT_b^{-1})^1f $ denotes the first component $\rho$ of the preimage $(\rho, \alpha, \beta)= \mmT_b^{-1}f$. Since $\gamma>4-1/p$ the following diagram, together with Corollary \ref{cor:smoothTb} (with $\gamma-2>2-1/p$) and  Proposition \ref{prop:FredholmCompDerivatives},  shows that the composition \[
\mmT_b^{-1}S \partial_x( \mmT_b^{-1})^1: L^p_{\gamma} \rightarrow M^{2,2}_{\gamma-4}\times \bbR^2
,\]
is bounded for all $b\geq0$.
\begin{equation} \label{diagram:Tb1}
\begin{tikzpicture}
\matrix(m)[matrix of math nodes,
 row sep=3.5em, column sep=5.5em,
  text height=2.5ex, text depth=0.25ex]
   {L^p_{\gamma}& M^{2,p}_{\gamma-3}& M^{1,p}_{\gamma-2}  & M^{2,p}_{\gamma-4} \times \bbR^2,  \\};
    \path[->,font=\footnotesize,>=angle 90] 
    	(m-1-1) edge node[auto] {$(\mmT_b^{-1})^1$} (m-1-2)
	(m-1-2) edge node[auto] {$S  \partial_x $} (m-1-3)
	(m-1-3) edge node[auto] {$\mmT^{-1}_b $} (m-1-4);
	\end{tikzpicture}
\end{equation}

We next show differentiability, 
\[ \left \| (R|_{b+h}  - R|_b)f - h \partial_b R|_b f \right\|_{M^{2,p}_{\gamma-4} \times \bbR^2} = \rmO(h^2).\]
Indeed, we can bound the left-hand side by
\begin{align*}
\left \| (R|_{b+h}  - R|_b)f - h \partial_b R|_b f \right\|_{M^{2,p}_{\gamma-4} \times \bbR^2}  = & \left \| 2h \mmT_b^{-1} S \partial_x (\mmT_{b+h}^{-1})1f  + 2h \mmT_b^{-1} S \partial_x(\mmT_b^{-1})^1 f  \right\|_{M^{2,p}_{\gamma-4} \times \bbR^2} \\
&\leq 2|h|\; \left \|\mmT^{-1}_bS \partial_x \right \|_{M ^{2,p}_{\gamma-3}\rightarrow M ^{2,p}_{\gamma-4}\times\R^2} \; \left \| \left( \mmT^{-1}_{b+h}- \mmT^{-1}_b \right)^1 f \right \|_{M ^{2,p}_{\gamma-3}}\\
&\leq 4|h|^2\; \left \|\mmT^{-1}_b S \partial_x \right \|_{M ^{2,p}_{\gamma-3}\rightarrow M ^{2,p}_{\gamma-4}\times\R^2} \; \|f  \|_{L^p_{\gamma}},
\end{align*}
where, because $\gamma> 4-1/p$, we are able to use  \eqref{eq:boundsrho} in the last inequality. 

Next, we show that the derivative $\partial_bR$ is continuous with respect to $b$ by proving that the following inequality holds
\[ 
\left\| \partial_bR|_{b+h}f - \partial_bR|_bf \right\|_{M^{2,p}_{\gamma-4} \times \bbR^2} \leq C |h|\,\|f\|_{L^p_{\gamma}}.
\]
We will use diagram \eqref{diagram:Tb1} along with the following modified version to justify the choice of spaces in each step. 
\begin{equation}\label{diagram:Tb2}
\begin{tikzpicture} 
\matrix(m)[matrix of math nodes,
 row sep=3.5em, column sep=5.5em,
  text height=2.5ex, text depth=0.25ex]
   {L^p_{\gamma}& M^{2,p}_{\gamma-2}& M^{1,p}_{\gamma-1}  & \mmX= M^{2,p}_{\gamma-4} \times \bbR^2.  \\};
    \path[->,font=\footnotesize,>=angle 90] 
    	(m-1-1) edge node[auto] {$(\mmT_b^{-1})^1$} (m-1-2)
	(m-1-2) edge node[auto] {$S  \partial_x $} (m-1-3)
	(m-1-3) edge node[auto] {$\mmT^{-1}_b $} (m-1-4);
	\end{tikzpicture}
\end{equation}
The triangle inequality, the continuity of the operator $\mmT_b^{-1} S \partial_x: M^{2,p}_{\gamma-2} \rightarrow M^{2,p}_{\gamma-4}\times \bbR^2$, and the continuity in $b$ of the operator $\mmT_b^{-1}:L^p_{\gamma-1} \rightarrow M^{2,p}_{\gamma-4} \times \bbR^2$ (since $\gamma -1>3-1/p$) show that
\begin{align*}
\left \|   \partial_bR|_{b+h}f \right.&\left.- \partial_bR|_bf    \right \|_{M^{2,p}_{\gamma-4} \times \bbR^2} \\
&= 2 \left \| \mmT_{b+h}^{-1} S \partial_x ( \mmT_{b+h}^{-1})^1 f -  \mmT_{b}^{-1} S \partial_x ( \mmT_{b}^{-1})^1 f \right \|_{M^{2,p}_{\gamma-4} \times \bbR^2}\\
&\leq 2 \;  \left[  \left \| \mmT^{-1}_{b+h} S \partial_x \left ( \mmT_{b+h}^{-1} - \mmT_b^{-1}\right )^1 f \right \|_{ M^{2,p}_{\gamma-4} \times \bbR^2  } + \left \| \left ( \mmT_{b+h}^{-1} - \mmT_b^{-1} \right )  ( S \partial_x (\mmT_b^{-1})^1f) \right \|_{M^{2,p}_{\gamma-4} \times \bbR^2}\right] \\
& \leq 2 \; \left[  \left \|\mmT^{-1}_{b+h} S \partial_x \right  \|_{M^{2,p}_{\gamma-3} \rightarrow M^{2,p}_{\gamma-4} \times \bbR^2} \; \left \| \left ( \mmT_{b+h}^{-1} - \mmT_b^{-1} \right)^1f \right \|_{M^{2,p}_{\gamma-3}} + 2|h| \left \| S \partial_x (\mmT_b^{-1})^1 f \right \|_{L^p_{\gamma-1}}\right].
\end{align*}
Then, since we also have continuity in $b$ of the operator $\mmT_b^{-1}:L^p_{\gamma} \rightarrow M^{2,p}_{\gamma-3} \times \bbR^2$, we obtain the desired result,
\begin{align*}
 \left \|   \partial_bR|_{b+h}f - \partial_bR|_bf    \right \|_{M^{2,p}_{\gamma-4} \times \bbR^2} & \leq 4\; |h| \left[  \left \|\mmT^{-1}_{b+h} S \partial_x \right \|_{M^{2,p}_{\gamma-3} \rightarrow M^{2,p}_{\gamma-4} \times \bbR^2} \;  \left \|f \right \|_{L^p_{\gamma}} + 2 \left \| S \partial_x (\mmT_b^{-1})^1 f \right \|_{M^{1,p}_{\gamma+1}} \right]\\
 & \leq 4\; |h| \left[  \left \|\mmT^{-1}_{b+h} S \partial_x \right \|_{M^{2,p}_{\gamma-3} \rightarrow M^{2,p}_{\gamma-4} \times \bbR^2} \;  \left \|f \right \|_{L^p_{\gamma}} + 2 \left \| S \partial_x (\mmT_b^{-1})^1\right \|_{L^p_{\gamma} \rightarrow M^{1,p}_{\gamma-1}}\; \| f  \|_{L^p_{\gamma}} \right].
\end{align*}
\end{proof}

\subsection{Differentiability of  $b\mmT_b^{-1}: L^p_{\gamma} \rightarrow M^{2,p}_{\gamma-3}\times \bbR^2$ for $\gamma>3-1/p$}\label{subsec:ContinuousbTb} 
Unfortunately, we will need differentiability of $\mmT_b^{-1}$ in a stronger topology than  the one used in the previous section. However, we can exploit that fact that the dangerous terms  carry an additional factor $b$. The following two lemmas show that the extra factor $b$ allows us to gain one degree of localization relative to the previous results. 

\begin{lem}\label{lem:contbTb}
Let $\gamma>2-1/p$ for $p \in (1, \infty)$. Then $b\mmT_b^{-1}:L^p_{\gamma} \rightarrow M^{2,p}_{\gamma-2} \times \bbR^2$ is Lipshitz continuous in $b\geq 0$ in the operator topology.
\end{lem}
\begin{proof}
Similar to the preceding section, we pick $f\in L^p_{\gamma}$ and show that 
\[
\left \|  \left[ (b+h) \mmT_{b+h}^{-1} - b \mmT_b^{-1}\right] f \right\|_{M^{2,p}_{\gamma-2} \times \bbR^2} \leq C |h| \| f\|_{L^p_{\gamma}}
.\]
Equivalently, writing $( \rho, \alpha, \beta)|_b = b \mmT_b^{-1} f$ and $( \rho, \alpha, \beta)|_{b+h} = {b+h} \mmT_{b+h}^{-1} f$, we show that
\[ 
\left \| (\Delta \rho, \Delta \alpha, \Delta \beta) \right\|_{M^{2,p}_{\gamma -2} \times \bbR^2} \leq C |h| \|f\|_{L^p_{\gamma}}
.\]
First, notice that the difference $(\Delta \rho, \Delta \alpha, \Delta \beta)= ( \rho, \alpha, \beta)|_{b+h}-( \rho, \alpha, \beta)|_b$ solves the equation 
\[
\mmT_b(\Delta \rho, \Delta \alpha, \Delta \beta) = -2h S \partial_x \rho|_{b+h} + hf.
\]  
Since $\gamma> 2-1/p$, from Lemma \ref{lem:InvertibleTb} we know that the function $\rho|_{b+h}$ satisfies $\| \rho\|_{\mathcal{D} \times \bbR^2} \leq \dfrac{C}{b+h} \|(b+h) f\|_{L^p_{\gamma}}$, where $\mathcal{D} =\{ u \in M^{2,p}_{\gamma-2}: u_x \in L^p_{\gamma}\}$. Therefore,
\begin{equation}\label{eq:boundsbTb}
 \left \| (\Delta \rho, \Delta \alpha, \Delta \beta) \right\|_{M^{2,p}_{\gamma -2} \times \bbR^2} \leq C\|-2h S \partial_x \rho|_{b+h} + hf\|_{L^p_{\gamma}}\leq 3 C|h|\, \|f\|_{L^p_{\gamma}}. 
 \end{equation}
\end{proof}

\begin{lem}\label{lem:diffbTb}
Let $\gamma>3-1/p$ and $p \in (1, \infty)$.  Then  $b\mmT_b^{-1}:L^p_{\gamma} \rightarrow M^{2,p}_{\gamma-3} \times \bbR^2$ is  differentiable in $b\geq 0$ with Lipshitz continuous derivative, with values in the operator norm topology.
\end{lem}
\begin{proof}
We again write $R|_b= b\mmT_b^{-1}$ and introduce the definition of the candidate for the derivative, $\partial_b R|_b =  \mmT_b^{-1} +2b\mmT_b^{-1}S \partial_x ( \mmT_b^{-1})^1$.
Following the proof of  Lemma \ref{lem:diffTb}, and since $\gamma>3-1/p$, we find uniform bounds for this operator.  We next show that that for $f \in L^p_{\gamma}$ we have
\[ 
\left \| \left(R|_{b+h} - R_b\right) f - h \partial_bR|_b f   \right\|_{M^{2,p}_{\gamma-3} \times\bbR^2} = \rmO(h^2)
.\]
Using the fact that $\left [(b+h)\mmT^{-1}_{b+h} -b\mmT_b^{-1}\right ] f = b\mmT_b^{-1} f -2h\mmT_b^{-1} S \partial_x \mmT_{b+h}^{-1})^1 f $ we can rewrite the left-hand side of the above expression as
\[
 \left \| \left(R|_{b+h} - R_b\right) f - h \partial_bR|_b f   \right\|_{M^{2,p}_{\gamma-3} \times\bbR^2} = \left \| \left [(b+h)\mmT^{-1}_{b+h} -b\mmT_b^{-1}\right ] f - h \left [ \mmT_b^{-1} -2b\mmT_b^{-1}S \partial_x(\mmT_b^{-1})1 \right ]f \right \|_{M^{2,p}_{\gamma-3} \times\bbR^2}. \]
Now, recall the inequality \eqref{eq:boundsbTb}, which for $\gamma >2-1/p$ shows the continuity in $b$ of the operator $b \mmT_b^{-1}: L^p_{\gamma} \rightarrow M^{2,p}_{\gamma-2} \times \bbR^2 $. This result, together with the linearity of $\mmT_b^{-1}S \partial_x: M^{1,p}_{\gamma-2} \times \bbR^2 \rightarrow M^{2,p}_{\gamma-3} \times \bbR^2$ (for $\gamma-1>2-1/p$), shows that
\begin{align*}
 \left \| \left(R|_{b+h} - R_b\right) f -\right.&\left. h \partial_bR|_b f   \right\|_{M^{2,p}_{\gamma-3} \times\bbR^2}  
 \\
 \leq & \left \|  \left[ h\mmT_b^{-1} f - 2h\mmT_b^{-1} S \partial_x ( (b+h) \mmT_{b+h}^{-1})^1 f \right] - h\left[   \mmT_b^{-1}f -2\mmT_b^{-1}S \partial_x(b\mmT_b^{-1})^1f \right ]   \right\|_{M^{2,p}_{\gamma-3} \times \bbR^2}\\
  \leq &2|h| \left \| \mmT_b^{-1}S \partial_x   \left[ (b+h) \mmT_{b+h}^{-1} - b \mmT_b^{-1}\right]^1f        \right\|_{M^{2,p}_{\gamma-3} \times \bbR^2}\\
  \leq & C | h|^2 \left \|   \mmT_b^{-1}S \partial_x   \right\|_{M^{2,p}_{\gamma-2} \to M^{2,p}_{\gamma-3} \times \bbR^2}  \| f\|_{L^p_{\gamma}},
\end{align*} 
as desired. The final step is to prove that the derivative $\partial_b R$ is Lipshitz in $b$. For $f\in L^p_{\gamma}$ we show that
\[ \big \| (\partial_b R|_{b+h} - \partial_bR|_b ) f\big\|_{M^{2,p}_{\gamma-3} \times \bbR^2} \leq C |h|\,\|f\|_{L^p_{\gamma}}. \]
Using the triangle inequality we obtain a first bound,
\begin{align*}
 \big \| (\partial_b R|_{b+h} &- \partial_bR|_b ) f\big\|_{M^{2,p}_{\gamma-3} \times \bbR^2} \\
 \leq & \left \| \left [ \mmT_{b+h}^{-1} -2(b+h)\mmT^{-1}_{b+h} S \partial_x (\mmT^{-1}_{b+h})^1 \right  ]f -  \left [ \mmT_{b}^{-1} -2b\mmT^{-1}_{b} S \partial_x (\mmT^{-1}_{b})^1 \right ]f \right \|_{M^{2,p}_{\gamma-3} \times \bbR^2}\\
\leq & \; \left \| \left (\mmT_{b+h}^{-1} - \mmT_b^{-1}\right )f \right \|_{M^{2,p}_{\gamma-3} \times \bbR^2} +2 \; \left \|  \left[ (b+h)\mmT_{b+h}^{-1} - b \mmT_b^{-1} \right] \left ( S \partial_x(\mmT^{-1}_{b+h})^1 \right)f    \right \|_{M^{2,p}_{\gamma-3} \times \bbR^2} \\
&+ 2| b |\; \left \|  \mmT_b^{-1} S \partial_x \left( \mmT_{b+h}^{-1} - \mmT_b^{-1} \right)^1 f  \right \|_{M^{2,p}_{\gamma-3} \times \bbR^2}.
\end{align*}
Notice that because $\gamma-1>2-1/p$, we can use again inequality \eqref{eq:boundsbTb} to conclude that the operator $b\mmT_b^{-1}:L^p_{\gamma-1} \rightarrow M^{2,p}_{\gamma-3} \times \bbR^2$ is continuous in $b$. Furthermore, since $S \partial_x( \mmT_{b+h}^{-1})^1: L^p_{\gamma} \rightarrow M^{1,p}_{\gamma-1}$ is linear, we can bound the second term in this last inequality by
\begin{equation}\label{eq:bound1}
\left \|  \left[ (b+h)\mmT_{b+h}^{-1} - b \mmT_b^{-1} \right] \left ( S \partial_x(\mmT^{-1}_{b+h})^1 \right)f    \right \|_{M^{2,p}_{\gamma-3} \times \bbR^2} \leq C| h|\,\| S \partial_x (\mmT_{b+h}^{-1})1 f\|_{L^p_{\gamma-1}} \leq C |h|\,\| f\|_{L^p_{\gamma}}
\end{equation}
On the other hand, since $\gamma-1>2-/p$ the operator $ \mmT^{-1}_b S \partial_x: M^{2,p}_{\gamma-2} \rightarrow M^{2,p}_{\gamma-3} \times \bbR^2$ is bounded and we have that
\[\left \|  \mmT_b^{-1} S \partial_x \left( \mmT_{b+h}^{-1} - \mmT_b^{-1} \right)^1 f  \right \|_{M^{2,p}_{\gamma-3} \times \bbR^2} \leq \left \|  \mmT_b^{-1} S \partial_x \right \|_{M^{2,p}_{\gamma-2} \rightarrow M^{2,p}_{\gamma-3}}  \left\| \left( \mmT_{b+h}^{-1} - \mmT_b^{-1} \right)^1 f  \right \|_{M^{2,p}_{\gamma-2}}.  \]
In particular, for $\gamma> 2 -1/p$ Lemma \ref{lem:InvertibleTb} shows that
\[ \| (\mmT_{b+h}^{-1} - \mmT_b^{-1} ) f \|_{M^{2,p}_{\gamma-2} \times \bbR^2} \leq |h|\, \| S \partial_x \rho|_{b+h} \|_{L^p_{\gamma}}\leq \frac{C|h|}{b} \| f\|_{L^p_{\gamma}}.\]
Hence,
\begin{equation}\label{eq:bound2}
\left \|  \mmT_b^{-1} S \partial_x \left( \mmT_{b+h}^{-1} - \mmT_b^{-1} \right)^1 f  \right \|_{M^{2,p}_{\gamma-3} \times \bbR^2} \leq \frac{C}{b}|h|\, \| f\|_{L^p_{\gamma}}.
\end{equation}
Finally, since $M^{2,p}_{\gamma-2} \subset M^{2,p}_{\gamma-3}$, Lemma \ref{lem:contbTb} and the bounds \eqref{eq:bound1}--\eqref{eq:bound2} show that
\[  \big \| (\partial_b R|_{b+h} - \partial_bR|_b ) f\big\|_{M^{2,p}_{\gamma-3} \times \bbR^2}
\leq C |h|\,\|f\|_{L^p_{\gamma}},\]
which concludes the proof of the lemma.
\end{proof}
%

\subsection{Proof of Theorem \ref{thm:Eikonal}}\label{subsec:ProofEikonal}
We conclude the proof of Theorem \ref{thm:Eikonal}. The following proposition makes precise the way in which we will apply the Implicit Function Theorem.
\begin{prop}\label{prop:diffFb}
Under assumption {\bf H3} and with $\phi_1,a_1$, and $b_1$ as in Section \ref{subsec:O(1)}, there exists  $\eps_0>0$ such that the operator  $F_{\eps(b_1+\beta)} : (M^{2,2}_{\sigma}\times\bbR^2) \times [0,\eps_0)  \rightarrow M^{2,2}_{\sigma} \times \bbR^2$, defined by $$F_{\eps(b_1+\beta)}(\rho, \alpha, \beta;\eps) = [I - \eps \mmT^{-1}_{\eps (b_1+\beta)} ( \tilde{N}_1 + (b_1+\beta) \tilde{N}_2)] ( \rho, \alpha, \beta),$$
 is of class $C^1$ in $(\rho, \alpha, \beta)$ and  $\eps$. Moreover, its Fr\'echet derivative $D_{(\rho,\alpha, \beta)}F_{\eps(b_1+\beta)}$ is the identity at $( \rho, \alpha, \beta;\eps)=0$, hence invertible.
\end{prop}
In order to prove proposition \ref{prop:diffFb} we first show differentiability of the nonlinearity $\tilde{N}_1$.

\begin{lem}\label{lem:boundedNonlinearity}
The operator $\tilde{N}_1: M^{2,2}_{\sigma}\times \bbR^2 \rightarrow L^2_{\sigma+4}$ defined by 
\begin{align*}
 \tilde{N}_1(\rho, \alpha, \beta) =& (\partial_x\phi_1 + \partial_x \rho)^2 -(b_1+\beta)^2(1-S^2) +2(\partial_x \phi_1 + \partial_x \rho +(b_1 +\beta)S)(a_1 + \alpha +(b_1 +\beta)x)\partial_x S \\
 &+(a_1+ \alpha +(b_1 + \beta)x)^2(\partial_xS)^2 ,
 \end{align*}
 is smooth for $\sigma>2$ .
\end{lem}
\begin{proof}
Since $\tilde{N}_1$ is a bilinear Nemitsky operator, the result of the proposition follows once we show $\tilde{N}_1:M^{2,2}_{\sigma} \times \bbR^2 \rightarrow L^2_{\sigma+4}$ is well defined and bounded as a multilinear map. To that end, first notice that the terms involving $(1-S^2)$ and $\partial_xS$ are exponentially localized, hence belong to  $L^2_{\sigma+4}$.

It remains to show that  $(\partial_x \phi_1 + \partial_x \rho)^2 \in L^2_{\sigma+4}$. Since $\rho \in M^{2,2}_{\sigma}$ this term is of the form $f^2$, with  $f\in M^{1,2}_{\sigma+1}$.  In particular,  $f \cdot \langle x \rangle^{\sigma+1} \in W^{1,2}$, and, by Sobolev embeddings we have  $ f \cdot \langle x \rangle^{\sigma+1} \in C^0_\mathrm{b}$. Using $\sigma>2$ now gives the desired bound,
\[ \|f^2 \|_{L^p_{\sigma+4}} \leq \|f\langle x \rangle^{\sigma+2}\|_{C^0_\mathrm{b}} \| f\|_{L^2_{3}}\leq \|f\|_{M^{1,2}_{\sigma+1}} \| f\|_{M^{1,2}_{\sigma+1}}.\]
\end{proof}

\begin{proof}[of Proposition \ref{prop:diffFb}]
The results from Subsections \ref{subsec:ContinuousTb} and \ref{subsec:ContinuousbTb}  (choosing $\gamma= \sigma+2>4-1/2$), as well as Lemma \ref{lem:boundedNonlinearity}, and the fact that $\tilde{N}_2= 2S \partial_x \phi_1 \in L^2_{\sigma+3}$, show that the compositions 
\begin{enumerate}
\item  $\mathcal{N}_1=\mmT_{\eps (b_1+\beta)}^{-1}\tilde{N}_1(\rho,\alpha,\beta): M^{2,2}_{\sigma} \times \bbR^2 \rightarrow M^{2,2}_{\sigma} \times \bbR^2$, and
\item  $\mathcal{N}_2=\eps  (b_1+\beta)\mmT_{\eps  (b_1+\beta)}^{-1}\tilde{N}_2(\rho,\alpha,\beta):M^{2,2}_{\sigma} \times \bbR^2 \rightarrow M^{2,2}_{\sigma} \times \bbR^2$,
\end{enumerate}
are continuously differentiable in a neighborhood of the origin. It is here that we encounter the strong localization of the inhomogeneity stated in hypothesis (H3), which allow us to obtain $\tilde{N}_2 \in L^2_{\sigma+3}$ and use the results from Section  \ref{subsec:ContinuousbTb} . Here, we are using $\eps>0$ and $b_1>0$, so that $\eps(b_1+\beta)\geq 0$ for $\eps\geq 0$ and $|\beta|<b_1$.  Inspecting the dependence on $\eps$, we also readily conclude continuous differentiability in $\eps\geq 0$. At $\eps=0$, we only have the identity map, which is bounded invertible, so that the Implicit Function Theorem can be applied near the trivial root $(\rho,\alpha,\beta;\eps)=0$. 
\end{proof}

\section{Non-local Array of Oscillators}\label{section:Nonlocal}
We now return to the problem of nonlocal coupling, \eqref{e:eiknl}, and the proof of Theorem \ref{thm:Nonlocal}. Throughout, we will assume that $\int J=J_0=1$, and $\int x^2G(x)=G_2=1$, possibly after rescaling $x$ and $\phi$. 

We will see that under Hypothesis (H1), the linearization of  \eqref{e:eiknl} about the constant solution has similar Fredholm properties as the second derivative. As in the case of the eikonal equation in Section \ref{section:Eikonal}, we look for solutions of the form
\begin{equation}\label{eq:Ansatz}
\phi(x) = \tilde{\phi}(x) + a \tanh(x) + b x \tanh(x) -b^2t.
\end{equation}
and analyze the resulting equation, dropping tildes
\begin{align} \label{eq:complete}
0 =& (-I + G \ast) ( {\phi} + a \tanh(x) + b x \tanh(x)) + \eps g(x) -2b( J' \ast {\phi})(J' \ast x \tanh(x)) \\  \nonumber
& -2a(J' \ast {\phi})(J' \ast \tanh(x)) -  2ab(J' \ast \tanh(x))(J' \ast x \tanh(x))  \\ \nonumber
&- b^2[ (J' \ast x \tanh(x))^2 -1] - a^2(J' \ast \tanh(x))^2 - (J' \ast {\phi})^2.  
\end{align}

The remainder of this section is organized as follows. In Section \ref{subsec:PropertiesKernels}, we collect properties of the convolution kernels $G$ and $J$ and establish some elementary properties of the associated linear operators. Section \ref{subsec:Setup} sets up the proof, introducing first-order approximations and the nonlinear equation that we will solve using the Implicit Function Theorem. We start the proof of our main result, Theorem \ref{thm:Nonlocal}, in Section \ref{subsec:ProofMain}, subject to several propositions that establish smoothness of nonlinearities and linear operators, which are provided in  Sections \ref{subsec:Decomposition} and \ref{subsec:Differentiability}.

\subsection{ Properties of the convolution kernels $G$ and $J$} \label{subsec:PropertiesKernels}

The linear part, $-I+G\ast$, represents nonlocal diffusive coupling in the following sense. Consider the Fourier symbol $-1+\hat{G}(\ell)$, which is an analytic function on $\ell\in \R\times \rmi (-\delta,\delta)$, for some $\delta$ sufficiently small, due to the exponential localization of $G$. Moreover, by (H1),
\begin{equation}\label{e:G1}
-1+\hat{G}(\ell)< 0, \mbox{ for }\ell\neq 0 \quad -1+\hat{G}(\ell)=-\frac{1}{2}G_2\ell^2+\rmO(\ell^4),\qquad -1+\hat{G}(\ell)=-\ell^2 \hat{G}_b(\ell),
\end{equation}
with $G_b(x)$ exponentially localized, continuously differentiable, and $0 \neq \hat{G}_b(\ell)=:\hat{G}_{-1}^{-1}(\ell)$. The Fourier multiplier $\hat{G}_{-1}(\ell)$ gives rise to an order-two pseudo-differential operator and we formally write 
\[
G_{-1}*u=(1-\partial_{xx})(\tilde{G}*u),\qquad G_{-1}*(G_b*u)=u.
\]
Here $\tilde{G}$ is an order zero pseudo-differential operator.

Analyticity and exponential localization of $\hat{G}$ give uniform exponential decay of derivatives, which then readily implies bounded mapping properties in algebraically localized spaces, which we summarize below. 

%
%
%
\begin{lem} \label{lem:boundedG-I}
The convolution operators  $(-I + G ) : H^2_{\gamma} \rightarrow H^2_{\gamma}$ and  $(-I + G ) : M^{2,2}_{\gamma}\rightarrow H^2_{\gamma+2}$ are bounded for all $\gamma\geq 0$.
\end{lem}
%
We note that the inverse of $(-I+G)$ is unbounded, due to the vanishing Fourier symbol at $\ell=0$. We therefore introduced the kernel $G_b$ in \eqref{e:G1} through its Fourier symbol. Considerations analogous to Lemma \ref{lem:boundedG-I} give the following result. 
%
%

\begin{lem}\label{lem:boundedpseudoinverse}
The convolution operator $G_b: L^2_{\gamma} \rightarrow H^2_{\gamma} $ is an isomorphism for all  $\gamma\geq 0$.
\end{lem}

Similar statements also hold for the convolution operator $J'*u$. Since $J$ is twice continuously differentiable and exponentially localized, we find bounded mapping properties between algebraically localized spaces while gaining one derivative. 
\begin{lem}\label{lem:boundedJ}
The convolution operator $J': L^2_{\gamma} \rightarrow H^1_{\gamma}$ is bounded for all $\gamma\geq 0$.
\end{lem}

\subsection{Leading-order Ansatz and linear preconditioning} \label{subsec:Setup}
We are interested in finding steady solutions to equation \eqref{eq:complete}. Equivalently, we want to  find zeros of the operator defined by its right-hand side. From the previous section, we know that the linear part, $-I+G\ast$, can be written as the local operator $\partial_{xx}$, up to an invertible convolution operator $G_b$. Preconditioning with the inverse, $G_{-1}$, we therefore find a local linear part, but a now slightly more complicated, nonlocal nonlinearity. We will see that the basic strategy of the proof of Theorem \ref{thm:Eikonal} is still applicable. We first find leading-order approximations to the solutions of equation \eqref{eq:complete} using the properties of the Laplace operator in Kondratiev spaces.

\paragraph{First order approximations.}
We scale $\phi = \eps \phi_1, a = \eps a_1, b = \eps b_1$, and find from \eqref{eq:complete} at $\rmO(\eps)$,
\begin{equation}\label{eq:laplaceNonlocal}
 0= \partial_{xx}\phi_1 + a_1 \partial_{xx}S + b_1 \partial_{xx}(xS) + G_{-1} \ast g, \quad S= \tanh(x).
 \end{equation}
The results from Lemma \ref{lem:InvertibleLaplace}, together with Lemma \ref{lem:boundedpseudoinverse} and our assumption that the function $g$ is in the space $ H^2_{\sigma+4}$, show that solutions to equation \eqref{eq:laplaceNonlocal} satisfy $\phi_1 \in M^{2,2}_{\sigma+2}$ and 
\[ 
 a_1 = \frac{1}{2} \int x (G_{-1} \ast g)\;\rmd x=\frac{g_1}{G_2}, \quad b_1 = - \frac{1}{2} \int G_{-1} \ast g \; \rmd x= -\frac{g_0}{G_2}.
\]
As we announced earlier, we will set $G_2=J_0=1$, from now on. 
 
\paragraph{Solution Ansatz.}
We set $\phi = \eps( \phi_1 + \rho), \  a= \eps(a_1 + \alpha),\  b=\eps(b_1 + \beta),$ and insert this Ansatz into \eqref{eq:complete}. Applying the pseudo-differential operator $G_{-1}$ and dividing by $\eps$ gives
\begin{equation} \label{ eq:higherorder}
0=\tilde{F}_\eps( \rho, \alpha, \beta) :=  \mmT_{\eps (b_1+\beta)}(\rho, \alpha, \beta) - \eps {N}_1(\rho, \alpha,\beta) -2 \eps (b_1+\beta) {N}_2( \rho),
\end{equation}
where
\begin{align}\label{eq:nonlinearity}
{N}_1(\rho, \alpha,\beta) =&G_{-1} \ast \tilde{N}_1( \phi_1 + \rho,a_1 + \alpha, b_1+ \beta),\nonumber \\
\tilde{N}_1(\phi,a,b) = &b^2[(J' \ast xS)^2-1]  + a^2(J' \ast S)^2 + (J' \ast \phi)^2 + 2a( J' \ast \phi)(J' \ast S) \\ \nonumber
&+ 2ab(J'\ast S)(J' \ast xS) + 2b(J' \ast \phi)(J' \ast xS - S),
\end{align}
and 
\begin{align} \nonumber {N}_2( \rho ) =&  G_{-1} \ast ( S J' \ast( \phi_1 + \rho)) - S \partial_x \rho\\ \label{eq:nonlinearity2}
=&G_{-1}(S J' \ast \phi_1) + (G_{-1} -\delta)\ast (S J' \ast \rho) + S( J -\delta) \ast \partial_x \rho.
\end{align}

Here $\delta$ represents the Dirac delta distribution.\\
Preconditioning with the linear part, we may rewrite the equation $\tilde{F}_\eps( \rho, \alpha, \beta) =0$ as
\begin{equation*}
0=F_{\eps(b_1+\beta)}(\rho, \alpha, \beta;\eps) := [I -\eps \mmT_{\eps (b_1+\beta)}^{-1}( \tilde{N}_1 +2(b_1+\beta){N}_2)]( \rho, \alpha, \beta).
\end{equation*}
We look at $F_{\eps(b_1+\beta)}$ as a nonlinear map to which we would like to apply the Implicit Function Theorem near the trivial solution $(\rho,\alpha,\beta;\eps)=0$.  In particular, we need to show that, for $\sigma>2$, both, 
\begin{itemize}
\item[(i)] $T_{\eps (b_1+\beta) }^{-1}N_1:M^{2,2}_{\sigma} \times \bbR^2 \rightarrow M^{2,2}_{\sigma} \times \bbR^2$, and
\item[(ii)] $\eps  (b_1+\beta) T_{\eps  (b_1+\beta)}^{-1} N_2:M^{2,2}_{\sigma} \times \bbR^2 \rightarrow M^{2,2}_{\sigma} \times \bbR^2$,
\end{itemize}
are $C^1$ in  $\eps,\rho,\alpha, \beta$, for $\eps\geq 0$, and using that $b_1>0$. 
 
\subsection{Proof of Theorem \ref{thm:Nonlocal}} \label{subsec:ProofMain}
We first  concentrate  on the operator $T_{\eps (b_1+\beta) }^{-1}N_1$ from  $(i)$. In Lemma \ref{lem:SmoothN1} we show that $N_1:M^{2,2}_{\sigma} \times \bbR^2 \rightarrow L^2_{\sigma+4}$ is smooth. We then use the fact that for $\sigma+4>4-1/2$ the operator $T_b^{-1}:L^2_{\sigma+4} \rightarrow M^{2,2}_{\sigma}$ is linear and $C^1$ in $b\geq 0$; see  Section \ref{section:Eikonal}.

\begin{lem}\label{lem:SmoothN1}
For $\sigma>2$, the operator $N_1:M^{2,2}_{\sigma} \times \bbR^2 \rightarrow L^2_{\sigma+4}$, defined in (\ref{eq:nonlinearity}), is continuously differentiable in a neighborhood of the origin.
\end{lem}
\begin{proof}
Since $G_{-1}:H^2_{\sigma+4}\to L^2_{\sigma+4}$ is bounded and $\tilde{N}_1$ quadratic, it is sufficient to show that $\tilde{N}_1:M^{2,2}_{\sigma} \times \bbR^2 \rightarrow H^2_{\sigma+4}$ is bounded. This is immediately clear for all terms except for 
$(J' \ast (\phi_1 + \rho))^2$, due to the exponential localization of $S'$ and $(1-S^2)$. Since $J$ defines a bounded convolution operator, and since 
$J' \ast (\phi_1 + \rho)=J \ast (\partial_x\phi_1 + \partial_x\rho)$, boundedness of this remaining term follows as in Lemma \ref{lem:boundedNonlinearity}. 
%
%
 \end{proof}

To show  continuous differentiability of  $\eps(b_1+\beta)T_{\eps(b_1+\beta)}^{-1}N_2(\rho):M^{2,2}_{\sigma} \times \bbR \rightarrow M^{2,2}_{\sigma} \times \bbR^2$, we  first decompose,
\[
N_2(\rho) = D\bar{N}_2 (\rho) + G_{-1} \ast S(J' \ast \phi_1), \mbox{ where } D = \partial_x(1-\partial_x)^{-1},
\]
where $\bar{N}_2$ will be made explicit, later.  We then show that $\bar{N}_2$ is $C^1$, Section \ref{subsec:Decomposition},  and that $bT_b^{-1}D:L^2_{\sigma+2} \rightarrow M^{2,2}_{\sigma }\times\bbR^2$ is $C^1$ in $b\geq 0$, with values in the space of operators with norm topology; Section  \ref{subsec:Differentiability}.

Now, hypothesis (H3) implies that the term $\phi_1 \in M^{2,2}_{\sigma+2}$ and so   $G_{-1} \ast S(J' \ast \phi_1)\in L^2_{\sigma+3}$. Then using Lemma \ref{lem:diffbTb} with $\gamma= \sigma+3$ shows that $\eps(b_1+\beta)T_{\eps(b_1+\beta)}^{-1}G_{-1} \ast S(J' \ast \phi_1)$ is $C^1$ in $\beta$. 

Summarizing, we need to show
\begin{itemize}
\item[(i)] $\bar{N}_2: M^{2,2}_{\sigma} \rightarrow L^2_{\sigma+2}$ is continuously differentiable; see results from Section \ref{subsec:Decomposition} with $\gamma=\sigma+ 2$;
\item [(ii)] $bT_b^{-1}D: L^2_{\sigma+2} \rightarrow M^{2,2}_{\sigma} \times \bbR^2$ is continuously differentiable in $b\geq 0$; see results from Section \ref{subsec:Differentiability}, with $\gamma=\sigma+2$.
\end{itemize}

Theorem \ref{thm:Nonlocal} then follows in a completely analogous fashion to the proof of Theorem \ref{thm:Eikonal}.

%
%
%
%
%
%
%
 
 \subsection{Decomposition of $N_2(\rho)$ and smoothness of $\bar{N}_2: M^{2,2}_{\sigma} \rightarrow L^2_{\sigma+2}$, for $\sigma+2>0$}\label{subsec:Decomposition}
 
We first recall the definition of $N_2(\rho$), 
\[
N_2(\rho) =  G_{-1}(S J' \ast \phi_1) + (G_{-1} -\delta)\ast (S J' \ast \rho) + S( J -\delta) \ast \partial_x \rho.\]
 We will next show that $(J -\delta)=DJ_2$, $(G_{-1}-\delta)=DG_{-2}$, with $D= \partial_x(1-\partial_x)^{-1}$, and establish operator norm bounds on $J_2$ and $G_{-2}$. Preparing for the proof, notice that, for  $f \in M^{s,2}_{\gamma}$, 
\[\hat{f} \in H^{\gamma}, \quad k\hat{f} \in H^{\gamma+1}, \quad \cdots, \quad k^s\hat{f}\in H^{\gamma+s}
.\] 
\begin{lem}\label{lem:boundedJ-d}
For $\gamma>0$, the convolution operator $(J -\delta)$ can be written as the (commutative) product of  $\partial_x(1- \partial_x)^{-1}$ and $J_2$, where $J_2: H^1_{\gamma} \rightarrow H^1_{\gamma}$ is bounded. In particular, the composition 
\begin{equation*}
\begin{tikzpicture}
\matrix(m)[matrix of math nodes,
 row sep=3.5em, column sep=5.5em,
  text height=2.5ex, text depth=0.25ex]
   {M^{1,2}_{\gamma-1} & H^1_{\gamma}& H^1_{\gamma}, \\};
    \path[->,font=\footnotesize,>=angle 90] 
    	(m-1-1) edge node[auto] {$ \partial_x(1-\partial_x)^{-1}$} (m-1-2)
	(m-1-2) edge node[auto] {$ J_2 $} (m-1-3);
	\end{tikzpicture}
\end{equation*}
 is bounded.
\end{lem}
\begin{proof}
Boundedness of $\partial_x(1-\partial_x)^{-1}$ was shown in 
Proposition \ref{prop:FredholmCompDerivatives}. 
To show boundedness of $J_2$, we prove the result for $\gamma$ integer, and conclude the general result by interpolation.  We define 
\[
\hat{J}_2(k) = \dfrac{1-ik}{ik}(\hat{J}(k)-1).
\]
Since we normalized $J_0=1$,  $\hat{J}(k)= 1+ \rmO(k^2)$, so that  $\hat{J}_2$ is analytic and decays at infinity. \footnote{Since $\hat{J}_2$ is of order $\rmO(k)$ near the origin, we could improve the result slightly, here.}
 
Now, suppose $f\in L^2_\gamma$. The properties of $\hat{J}_2$ then imply that $\hat{J}_2 \hat{f} \in H^{\gamma}$ and therefore $J_2:L^2_\gamma \rightarrow L^2_\gamma$ is a bounded convolution operator. Since the convolution commutes with derivatives, we also conclude boundedness on $H^1_\gamma$. 
\end{proof}
 
We consider the pseudo-differential operator operator $G_{-1}-\delta$, next.
\begin{lem}\label{lem:boundedG-d}
Let $\gamma>0$ then, the convolution operator $(G_{-1}-\delta)$ can be written as $(G_{-1}-\delta) = \partial_x(1-\partial_x)^{-1} G_{-2}$, where $G_{-2}:H^2_{\gamma} \rightarrow L^2_{\gamma}$ is bounded. In particular, the composition
\begin{equation*}
\begin{tikzpicture}
\matrix(m)[matrix of math nodes,
 row sep=3.5em, column sep=5.5em,
  text height=2.5ex, text depth=0.25ex]
   {H^2_{\gamma} & L^2_{\gamma}& L^2_{\gamma}, \\};
    \path[->,font=\footnotesize,>=angle 90] 
    	(m-1-1) edge node[auto] {$G_{-2} $} (m-1-2)
	(m-1-2) edge node[auto] {$ \partial_x(1-\partial_x)^{-1} $} (m-1-3);
	\end{tikzpicture}
\end{equation*}
 is bounded. 
\end{lem}
\begin{proof}
The proof is similar to the proof of Lemma \ref{lem:boundedJ-d}, exploiting that $\hat{G}_{-1}(0)=1$, noticing the normalization $G_2=1$. We omit the straighforward adaptation.
\end{proof}
As a corollary to the two preceding lemmas, we have established the following decomposition.

\begin{cor}\label{cor:decomposingN2}
Let $D= \partial_x(1-\partial_x)^{-1}$ then, we can write $$ N_2(\rho) = D \bar{N}_2(\rho) + N_3,$$ where 
\begin{enumerate}
\item the operator $\bar{N}_2:M^{2,2}_{\sigma} \rightarrow L^2_{\sigma+2}$ defined by $$\bar{N}_2 (\rho) = G_{-2}\ast(S J' \ast \rho)+ S J_2 \ast \partial_x\rho,$$ is bounded for $\sigma+2>0$;
\item  the constant  $N_3= G_{-1} \ast S(J' \ast \phi_1) + [S,D](J_2 \ast \partial_x\rho)$ lies in $L^2_{\sigma+2}$. In particular, the term $[S,D] (J_2 \ast \rho)$ is exponentially localized.
\end{enumerate}
\end{cor}
\begin{proof}
A straightforward calculation shows that the decomposition of $N_2(\rho)$ is as stated in the Corollary.
Item (i) follows from Lemmas \ref{lem:boundedJ-d} and \ref{lem:boundedG-d}. In particular, notice that, because $\rho \in M^{2,2}_{\sigma}$, the function $J'\ast \rho$ satisfies
\[ 
J' \ast \rho \in H^1_{\sigma}, \quad J' \ast \partial_x \rho \in H^1_{\sigma+1}, \quad J' \ast \partial_{xx} \rho \in H^1_{\sigma+2}
,\]
and, since the pseudo-differential operator with $G_{-2}$ is of order two, the term $G_{-2}\ast(S J' \ast \rho)$ belongs to  $L^2_{\sigma+2}$.
  
To establish (ii),  we only need to show that the commutator $[S,D](J_2 \ast \partial_x \rho)$ is exponentially localized, since it was shown already at the end of Subsection \ref{subsec:ProofMain} that the term $G_{-1} \ast S(J' \ast \phi_1) $ lies in $L^2_{\sigma+3}$. 
 
In what follows we use the fact that $(1-\partial_x)^{-1}:H^2_{\eta} \rightarrow L^2_{\eta}$ is a bounded invertible operator between exponentially weighted spaces\footnote{We use the subscript $\eta$ to denote exponential weights, and the subscript $\gamma$ to denote algebraic weights and indicate which weights are referred to when confusion is possible.}, $H^1_\eta$, where
\begin{equation}\label{e:expw}
H^s_{\eta}= \{ f \in L^2: f(x) \rme^{\eta \langle x\rangle} \in H^s, \quad \eta\in \bbR \}.
\end{equation}
Now, let $f= J_2 \ast \partial_x \rho$ and examine
\begin{align*}
[S,D]f =& S \partial_x (1-\partial_x)^{-1}f + \partial_x(1-\partial_x)^{-1} S f\\
(1-\partial_x)   [S,D]f =& \partial_xS (1-\partial_x)^{-1} \partial_x f - \partial_xS f.
\end{align*}
Since the right-hand side of this last equality belongs to $H^1_\eta$, invertibility of $(1-\partial_x)^{-1}$ implies that the commutator $[S,D] f $ is exponentially localized as well.
\end{proof}
\subsection{Differentiability of  $bT_b^{-1}D:L^2_{\gamma} \rightarrow M^{2,2}_{\gamma-2}$ for $\gamma>3/2$} \label{subsec:Differentiability}

To prove continuous differentiability of $bT_b^{-1}D: L^2_{\gamma} \rightarrow M^{2,2}_{\gamma-2} \times \bbR^2 $ we first establish Lipshitz continuity. Therefore, define 
\[
Y=M^{2,2}_{\gamma-2}\cap M^{1,2}_{\gamma-1},\qquad \| f\|_Y = \|f\|_{L^2_{\gamma-1}} + \|f_x \|_{H^1_{\gamma}}.
\]
Note that $D:M^{2,2}_{\gamma-2}\to Y$ is bounded.
\begin{lem}\label{lem:bTbDcontY}
Fix $\gamma>3/2$. Then, the operator $bT_b^{-1}D:L^2_{\gamma} \rightarrow Y\times \bbR^2$ is uniformly bounded and Lipshitz continuous in the parameter $b\geq 0$.
\end{lem}

\begin{proof}
We show Lipshitz continuity; uniform bounds can be established in a similar fashion. For $f \in L^2_{\gamma}$, we need to obtain the following estimate,
\[
\left \| (b+h)T^{-1}_{b+h} D f - bT_b^{-1} Df \right \|_{Y \times \bbR^2} \leq C |h|\,\|f\|_{L^2_\gamma}
.\]
We define the auxiliary function $g_b$, which will monitor commutators between $T_b$ and $D$. For this, first write
\[
T_b(\rho,\alpha,\beta)=\partial_{xx}\rho+2bS\partial_x\rho +\alpha\partial_{xx}S+\beta\partial_{xx}(xS)=bDf,
\]
and set 
\[
b\partial_x g_b  = 2b \partial_x(S \rho) -2b S \partial_x \rho + \alpha( 2 \partial_{xxx} S - \partial_{xx} S) + \beta(2 \partial_{xxx}(xS) - \partial_{xx}(xS) ).
\]
One readily notices that the right-hand side of this identity is exponentially localized and verifies that its average vanishes:
\begin{align*}
\int b \partial_x g_b &= - \int ( \partial_{xx} \rho + 2bS \partial_x \rho + \alpha \partial_{xx}S + \beta \partial_{xx}(xS)) \rmd x + 2 \left[ \int  \partial_{xxx} ( \alpha S + \beta xS) \rmd x \right]\\
&= - \int T_b( \rho, \alpha, \beta) \rmd x +  2 \left[ \int  \partial_{xxx} ( \alpha S + \beta xS) \rmd x \right]\\
&= - \int bD f \;\rmd x +  2 \left[ \int  \partial_{xxx} ( \alpha S + \beta xS) \rmd x \right].
\end{align*}
As a consequence, 
\[
bg_b = 2bS\rho-2b(\partial_x)^{-1}S\partial_x \rho + \alpha( \partial_{xx}S - D^{-1} \partial_{xx}S) + \beta( \partial_{xx}(xS) - D^{-1} \partial_{xx}(xS) ),
\]
is well defined and bounded in terms of $bT_b^{-1}Df$. A short calculation shows that 
\[
bg_b =  T_b(D^{-1}\rho, \alpha, \beta) - D^{-1}T_b( \rho , \alpha, \beta),
\]
so that 
\begin{align*}
(\rho,\alpha,\beta)_b &= bT_b^{-1} D f = bD^1T_b^{-1}(f + g_b)\\
(\rho,\alpha,\beta)_{b+h} &= (b+h)T_{b+h}^{-1} D f = (b+h)D^1T_{b+h}^{-1}(f + g_{b+h})
\end{align*}
where we used the shorthand  $D^1(\rho, \alpha, \beta) = (D\rho, \alpha, \beta)$.
The result of the lemma then follows if we can show that
\[ 
\left \| (b+h)D^1T^{-1}_{b+h} (f+g_{b+h}) - bD^1T_b^{-1}(f+g_b) \right\|_{Y \times \bbR^2} \leq C |h|\,\|f\|_{L^2_{\gamma}}
.\]
With this goal in mind, we we use the triangle inequality to obtain
\begin{align*}
 \left \| (b+h)D^1T^{-1}_{b+h} (f+g_{b+h}) - bD^1\right.&\left.T_b^{-1}(f+g_b) \right \|_{Y \times \bbR^2} \\
 \leq &\;  \left\| (b+h)D^1T^{-1}_{b+h} (f+g_{b+h}) - bD^1T_b^{-1}(f+g_{b+h}) \right \|_{Y \times \bbR^2}\\
&+\left \| bD^1T^{-1}_{b} (f+g_{b+h}) - bD^1T_b^{-1}(f+g_b) \right\|_{Y \times \bbR^2}.
\end{align*}
Uniform bounds on $D^1T^{-1}_b:L^2_{\gamma} \rightarrow Y \times \bbR^2$, which follow immediately from Corollary \ref{cor:smoothTb} (with $\gamma>2-1/p = 3/2$) and the definition of $Y$, allow us to simplify this estimate further,
\begin{align*}
 \left \| (b+h)D^1T^{-1}_{b+h} (f+g_{b+h}) -\right.&\left. bD^1T_b^{-1}(f+g_b) \right \|_{Y \times \bbR^2} \\
 \leq &\;  \left \| (b+h)D^1T^{-1}_{b+h} (f+g_{b+h}) - bD^1T_b^{-1}(f+g_{b+h}) \right \|_{Y \times \bbR^2}\\ 
 &+ C |b|\;  \| g_{b+h}-g_b \|_{L^2_{\gamma}}.
\end{align*}
The next step is to show that the operator $bD^1T_b^{-1}:L^2_{\gamma} \rightarrow Y \times \bbR^2$ is continuous in $b$ and that the difference  $g_b\in L^2_{\gamma}$ is Lipshitz in $b$.
To show the continuity in $b$ of  $bD^1T_b^{-1}$ we let 
\[\
( \rho, \alpha, \beta) |_b = bD^1 T_b^{-1} f, \quad  (\rho,\alpha, \beta)|_{b+h} = (b+h)D^1T_{b+h}^{-1} f, 
\]
and show the inequality,
\[ \|  ( \rho, \alpha, \beta) |_{b+h}- ( \rho, \alpha, \beta) |_b \|_{Y\times\bbR^2}\leq C |h| \|f\|_{L^2_{\gamma}}.\]
Letting $\psi = D^{-1} \rho$, the expression $( \rho, \alpha, \beta)= bD^1T_bf$ can be written as $T_b( \psi, \alpha, \beta) = b f$, and a short calculation shows that the difference between $(\psi,\alpha, \beta)|_b$ and $(\psi, \alpha, \beta)|_{b+h}$, denoted by $(\Delta \psi, \Delta \alpha, \Delta \beta)$, satisfies the equations $$T_b(\Delta \psi, \Delta \alpha, \Delta \beta)= hf - 2hS\partial_x \psi|_{b+h}.$$
Now, since the function $\psi|_{b+h}$ is a solution to $T_{b+h}( \psi, \alpha, \beta) = (b+h)f$, the results from Lemma \ref{lem:InvertibleTb} with $\gamma >1-1/p=1/2$ and $\mathcal{D}=\{ u \in M^{2,2}_{\gamma-2} : u_x \in L^2_{\gamma}\}$, show that
\[\| \psi \|_{\mathcal{D}} \leq \frac{C}{|b+h|} \| (b+h) f\|_{L^2_{\gamma}}.\]
In particular, we see that $\partial_x \psi|_{b+h}$ is in the space $L^2_{\gamma}$ and therefore, solutions to $$T_b(\Delta \psi, \Delta \alpha, \Delta \beta)= hf - 2hS\partial_x \psi|_{b+h},$$ with $\gamma>2-1/p= 3/2$, satisfy the inequality
\begin{equation}
\| ( \Delta \psi, \alpha, \Delta \beta)\|_{M^{2,2}_{\gamma-2} \times \bbR^2} \leq C |h|\,\| f - 2S\partial_x \psi|_{b+h}\|_{L^2_{\gamma}}\leq Ch \| f\|_{L^2_{\gamma}}.
\end{equation}
Lastly, because the operator $D^{-1}$ is linear, the difference $\Delta \psi = D^{-1} \Delta \rho$ is  in  $M^{2,2}_{\gamma-2}$ and we may conclude that $\Delta \rho \in Y$. Therefore,
\[  \| \Delta \rho, \Delta \alpha, \Delta \beta)\|_{Y\times \bbR^2}\leq C h\| f\|_{L^2_{\gamma}},\]
as desired.
Finally, we show that 
\[
|b|\,\| g_{b+h}-g_b\|_{L^2_{\gamma}} \leq C|h|\,\|f\|_{L^2_{\gamma}}.
\]
Notice that, writing $\psi_b= D^{-1} \rho_b$,
\[ 
T_b( \psi_b, \alpha_b, \beta_b) = b(f +g_b), \qquad T_{b+h}( \psi_{b+h}, \alpha_{b+h}, \beta_{b+h}) = (b+h)(f +g_{b+h})
.\]
Subtracting both equations and using the triangle inequality, we find that
\begin{align*}
| b|\,\| g_{b+h}- g_b\|_{L^2_{\gamma}} \leq& h \|f + g_{b+h}\|_{L^2_{\gamma}} + \| T_{b+h}( \psi_{b+h}, \alpha_{b+h}, \beta_{b+h}) -T_b( \psi_b, \alpha_b, \beta_b)\|_{L^2_{\gamma}}\\
\leq &  h \|f + g_{b+h}\|_{L^2_{\gamma}} + \| T_{b+h}( \psi_{b+h}, \alpha_{b+h}, \beta_{b+h}) -T_b( \psi_{b+h}, \alpha_{b+h}, \beta_{b+h})\|_{L^2_{\gamma}} \\
& + \| T_{b}( \psi_{b+h}, \alpha_{b+h}, \beta_{b+h}) -T_b( \psi_b, \alpha_b, \beta_b)\|_{L^2_{\gamma}}  \\
\leq &  h \|f + g_{b+h}\|_{L^2_{\gamma}} + h \| 2S \partial_x\psi_{b+h} \|_{L^2_{\gamma}} + \|h(f + 2S\partial_x \psi_{b+h})\|_{L^2_{\gamma}}  \\
\leq &  h \|f + g_{b+h}\|_{L^2_{\gamma}} + h \| 2S \partial_x\psi_{b+h} \|_{L^2_{\gamma}} + h C   \| f\|_{L^2_{\gamma}} \\
\leq & h \| f\|_{L^2_{\gamma}},
\end{align*}
where we used that  $\| \partial_x\psi_b\|_{L^2_{\gamma}} \leq \| \psi_b\|_{\mathcal{D}} \leq \|f+g\|_{L^2_{\gamma}}$ from Lemma \ref{lem:InvertibleTb} ( $\gamma >1/2$). This completes the proof.
\end{proof}

We are now ready to show the differentiability of the operator $bT_b^{-1} D:L^2_{\gamma} \rightarrow M^{2,2}_{\gamma-2} \times \bbR^2$.

\begin{lem}\label{lem:bTbDdifferentiable}
Fix $\gamma>3/2$. Then the operator $bT_b^{-1}D:L^2_{\gamma} \rightarrow M^{2,2}_{\gamma-2} \times \bbR^2$ is  differentiable in the parameter $b\geq 0$ with Lipshitz continuous derivative.
\end{lem}
\begin{proof}
We abbreviate $R=bT_b^{-1}D$ and recall the notation $(T_b^{-1})^1f$ for the  first component of $T_b^{-1}f$. We first define the candidate for the derivative, 
\[
\partial_b R|_b f = 2bT_b^{-1} S \partial_x (T_b^{-1})^1Df + T_b^{-1} Df,
\]
and show that
\[ \left\| (R|_{b+h} - R|_b)f - h \partial_bR|_b f \right \|_{M^{2,2}_{\gamma-2} \times \bbR^2} = \rmO(h^2).\]
A short calculation shows that the difference $$R|_{b+h}f -R|_b f = (\rho, \alpha, \beta)|_{b+h} - (\rho, \alpha, \beta)|_b =(\Delta \rho, \Delta \alpha, \Delta \beta)$$ satisfies the equation $T_b(\Delta \rho, \Delta, \alpha, \Delta \beta) = -2h S \partial_x \rho|_{b+h} + h Df$. Therefore,
\begin{align*}
\left\| (R|_{b+h} - \right.&\left.  R|_b)f- h \partial_bR|_b f \right \|_{M^{2,2}_{\gamma-2} \times \bbR^2} \\
=& \left \| \left[ -2h(b+h) T_b^{-1} S \partial_x (T^{-1}_{b+h})^1Df + h T_b^{-1} Df\right]  - \left[- 2h bT_b^{-1}S \partial_x ( T_b^{-1})^1Df- h T_b^{-1}Df \right] \right \|_{M^{2,2}_{\gamma-2} \times \bbR^2} \\
=& |2h| \left \| T_b^{-1}S \partial_x \left[ (b+h)T_{b+h}^{-1} - b T_b^{-1}\right]^1 Df \right\|_{M^{2,2}_{\gamma-2} \times \bbR^2}\\
\leq & |2h| \left\| T_b^{-1}S \partial_x \right\|_{Y\to M^{2,2}_{\gamma-2}\times \R^2 } \left\|  \left[ (b+h)T_{b+h}^{-1} - b T_b^{-1}\right]^1 Df \right\|_{Y},
\end{align*}
where the last inequality follows from the continuity of the operators
\begin{equation*}
\begin{tikzpicture}
\matrix(m)[matrix of math nodes,
 row sep=3.5em, column sep=9.5em,
  text height=2.5ex, text depth=0.25ex]
   {L^2_{\gamma} & Y& L^2_{\gamma} &M^{2,2}_{\gamma-2} \times \bbR^2. \\};
    \path[->,font=\footnotesize,>=angle 90] 
    	(m-1-1) edge node[auto] {$\left[(b+h)T^{-1}_{b+h} -bT_b^{-1} \right]^1D  $} (m-1-2)
	(m-1-2) edge node[auto] {$ S\partial_x $} (m-1-3)
	(m-1-3) edge node[auto]{$ T_b^{-1} $}(m-1-4);
	\end{tikzpicture}
\end{equation*}
Using the results from Lemma \ref{lem:bTbDcontY}, where we showed that for $\gamma>3/2$ the operator $bT_b^{-1}D:L^2_{\gamma} \rightarrow Y \times \bbR^2$ is continuous with respect to the parameter $b$, we see that
\begin{equation}\label{eq:boundedComp}
\left\| (R|_{b+h} - R|_b)f  - h \partial_bR|_b f \right \|_{M^{2,2}_{\gamma-2} \times \bbR^2} \leq  C|h^2|\,\|f \|_{L^2_{\gamma}}
\end{equation}
as desired.This proves differentiability. It remains to establish continuity of the derivative,
\[
\left \| ( \partial_bR|_{b+h} - \partial_b R|_b) f \right \|_{M^{2,2}_{\gamma-2} \times \bbR^2}\leq C|h|.
\]
We split the difference into
\begin{align*}
\left \| ( \partial_bR|_{b+h} \right.&\left.- \partial_b R|_b) f \right \|_{M^{2,2}_{\gamma-2} \times \bbR^2}\\
=& \left \| (b+h) T_{b+h}^{-1} S \partial_x (T^{-1}_{b+h})^1Df -   b T_{b}^{-1} S \partial_x (T^{-1}_{b})^1Df  \right \|_{M^{2,2}_{\gamma-2} \times \bbR^2}\\
\leq & 2 |b+h| \left \| \left[ T_{b+h}^{-1} -T_b^{-1}\right] \left ( S \partial_x (T^{-1}_{b+h}\right )^1 Df \right\|_{M^{2,2}_{\gamma-2} \times \bbR^2} + 2 \left \|  T_b^{-1} S \partial_x \left[ (b+h)T_{b+h}^{-1} - bT_b^{-1}\right] ^1 Df \right\|_{M^{2,2}_{\gamma-2} \times \bbR^2}\\
\leq &2|b+h| \frac{C\, |h|}{|b+h|} \left \| S \partial_x (T^{-1}_{b+h})^1 Df   \right\|_{L^2_{\gamma} } + 2 \left \| T_b^{-1}S \partial_x \right\|_{Y\to M^{2,2}_{\gamma-2}\times\R^2 } \left \|  \left[ (b+h)T_{b+h}^{-1} - bT_b^{-1}\right] ^1 Df \right\|_Y.
\end{align*}
This last inequality follows from the continuity of the operator $T_b^{-1} :L^2_{\gamma} \rightarrow \mathcal{D}$, see Lemma \ref{lem:InvertibleTb} with $\gamma>1/2$, and the boundedness of the composition \eqref{eq:boundedComp}. Then, using again the results from Lemma \ref{lem:bTbDcontY} with $\gamma>3/2$, we find that
\begin{equation}
\left \| ( \partial_bR|_{b+h} - \partial_b R|_b) f \right \|_{M^{2,2}_{\gamma-2} \times \bbR^2} \leq   2C | h|  \left \| S \partial_x (T^{-1}_{b+h})^1 Df   \right\|_{L^2_{\gamma} } + 2 h \left \| T_b^{-1}S \partial_x \right\|_{Y\to M^{2,2}_{\gamma-2}\times\R^2 }\| f\|_{L^2_{\gamma}},
\end{equation}
which shows Lipshitz continuity of the derivative and concludes the proof.\end{proof}

\section{Appendix}\label{sec:Appendix}
We prove Proposition \ref{prop:FredholmCompDerivatives}. Note first that it is sufficient to consider $\ell=0$, given Proposition \ref{prop:(1-dxx)}. 
The proof is split into several lemmas. We first consider the case $k=1$, $m=0$ and $\gamma>1-1/p$ , and establish Fredholm properties for 
\[
\partial_x: M^{1,p}_{\gamma-1} \rightarrow L^p_{\gamma}
.\]
We then treat the case $\gamma<1-1/p$ in a similar fashion and conclude by showing that $\partial_x$ is not Fredholm for $\gamma=1-1/p$. The results for general $m$ and $k$ follow easily from elemantary calculus and additivity of the index for composition of Fredholm operators. 
In what follows we will denote the subspace spanned by the constants as $\bbP_0$.
\begin{lem}
Let $p \in (1,\infty)$ and $\gamma > 1-1/p$. Then, the operator $$ \partial_x: M^{1,p}_{\gamma-1} \rightarrow L^p_{\gamma},$$ is Fredholm with index $-1$ and cokernel spanned by $\bbP_0$.
\end{lem}
\begin{proof}
Let $\gamma>1-1/p$ and let $C^{\perp}$ denote the orthocomplement of $\bbP_0$, that is \[
C^{\perp}=\{ f \in L^p_{\gamma}: \int f =0\}
,\]
a closed subspace of $L^p_\gamma$ given our restrictions on $\gamma$.
It is clear that this is a closed subspace of $L^p_{\gamma}$, since $1$ is a bounded linear functional on $L^p_{\gamma}$.

Notice first that  $\Rg(\partial_x) = C^{\perp}$, since any solution to the equation $\partial_xu=f$ solves 
\[
u(x) = \int_{-\infty}^x \partial_xfy) \rmd y,
\]
and $u(x)\to 0$ for $x\to\infty$. Also, the kernel of $\partial_x$ is trivial since constants do not belong to $M^{1,p}_{\gamma-1}$ with our restriction on $\gamma$. It remains to show that the inverse, given through the formula
\[  \begin{array}{c c c c}
  \partial_x^{-1}: & C^{\perp}& \rightarrow &M^{1,p}_{\gamma-1},\\
   & f& \mapsto &\int_{\infty}^x f(y) \rmd y=\int_{-\infty}^x f(y),
   \end{array}
\]
is bounded; note that both integration formulas differ by $\int_\R f=0$ since $f\in C^\perp$. Establishing the required bounds follows a strategy similar to the one used in  Lemma \ref{lem:FredholmLb}. We restrict to $x>0$ without loss of generality and show 
\[\|u\|_{L^p_{\gamma-1}(0,\infty)}< \frac{1}{\gamma -1 + 1/p} \|f\|_{L^p_{\gamma}(0,\infty)}.
\]
This is accomplished after substituting 
\begin{equation}\label{e:expsc}
x  = \rme^{\tau}, \tau\in\R, \quad  \qquad   w(\tau) = \rme^{\bar{\gamma} \tau} u( \rme^{\tau}), \qquad   g(\tau) = \rme^{(\bar{\gamma} +1 ) \tau} f( \rme^{\tau}),
\end{equation}
with $\bar{\gamma}=\gamma-1+1/p$, and estimating the exponential convolution kernel. We omit the straightforward details which are similar to the proof of Lemma \ref{lem:FredholmLb} but easier.
\end{proof}

\begin{lem}
Let $p \in (1,\infty)$ and $\gamma < 1-1/p$. Then, the operator $$ \partial_x: M^{1,p}_{\gamma-1} \rightarrow L^p_{\gamma},$$ is Fredholm with index $1$ and kernel spanned by $\bbP_0$.
\end{lem}
\begin{proof}
One readily verifies the claim on the kernel so that it is sufficient to verify that the operator is onto. Therefore, we define 
\[  \begin{array}{c c c c}
  \partial_x^{-1}: & L^p_\gamma& \rightarrow &M^{1,p}_{\gamma-1},\\
   & f& \mapsto &\int_{0}^x f(y) \rmd y.
   \end{array}
\]
Clearly, $\partial_x^{-1}$ is a right inverse. Using the coordinate transformations \eqref{e:expsc}, one obtains once again a convolution operator with exponentially localized kernel and finds the desired estimates in a straightforward fashion.
\end{proof}
Finally, we show that for $\gamma=1-1/p$, $\partial_x$ does not have closed range. 
\begin{lem}
Let $p \in (1,\infty)$ and $\gamma=1-1/p$. Then
\begin{equation*}
\partial_x : M^{1,p}_{\gamma-1} \rightarrow L^p_{\gamma}
\end{equation*}
does not have closed range.
\end{lem}
\begin{proof}
One readily finds that kernel and cokernel are trivial, yet the operator cannot be Fredholm since operators for nearby values of $\gamma$ can be viewed as compact perturbations, for which the Fredholm index jumps. As a consequence, the range cannot be closed.
\end{proof}

 \bibliographystyle{siam}	
\bibliography{nonloc-pacemaker-3}	
\end{document}